\documentclass[12pt]{amsart}
\usepackage{amssymb,latexsym, amsmath, amsthm,pdfsync}
\usepackage{enumerate}
\usepackage{layout}
\usepackage{setspace}
\usepackage{nicefrac}
\theoremstyle{plain}
\newtheorem{theorem}{Theorem}
\newtheorem{corollary}{Corollary}
\newtheorem{lemma}{Lemma}
\newtheorem{proposition}{Proposition}

\theoremstyle{definition}

\numberwithin{equation}{section}
\numberwithin{lemma}{section}
\numberwithin{theorem}{section}
\numberwithin{proposition}{section}
\renewcommand{\Re}{\mathrm{Re\,}}
\renewcommand{\Im}{\mathrm{Im\,}}
\newcommand{\bigO}{O}
\newcommand{\littleo}{o}
\begin{document}
\title[Bergman in $L^p$]{The Bergman projection in $L^p$
for domains with minimal smoothness}
\author[Lanzani and Stein]{Loredana Lanzani$^*$
and Elias M. Stein$^{**}$}
\thanks{$^*$ Supported in part by a  National Science Foundation IRD plan,  and by award DMS-1001304}
\thanks{$^{**}$ Supported in part by the National Science Foundation, award
DMS-0901040, and by KAU of Saudi Arabia}
\address{Dept. of Mathematics \\        
University of Arkansas 
Fayetteville, AR 72701
  USA}
\address{
Dept. of Mathematics\\Princeton University 
\\Princeton, NJ   08544-100 USA }
  \email{loredana.lanzani@gmail.com,\; stein@math.princeton.edu}
\begin{abstract} Let $D\subset\mathbb C^n$ be a bounded, strongly Levi-pseudoconvex domain with minimally smooth boundary.
We prove $L^p(D)$-regularity for the Bergman projection $B$, 
and 
for the operator $|B|$ whose kernel is the absolute value of the Bergman kernel with $p$ in the range $(1,+\infty)$.
 As an application,
we show that the space of holomorphic functions in a neighborhood of 
$\overline{D}$ is dense in $\vartheta L^p (D)$.
\end{abstract}
\maketitle
\centerline{{\em Dedicated to John P. D'Angelo, on the occasion of his 60th birthday}}
\medskip

\section{Introduction}

\indent This is the first in a series of papers dealing with the $L^p$-theory of reproducing operators such as the Cauchy-Fantappi\'e integrals and the 
Szeg\H o and Bergman projections for domains in $\mathbb{C}^n$ whose boundaries have minimal smoothness. In the present paper we study the Bergman projection on domains in $\mathbb{C}^n$ that are strongly 
pseudoconvex and have $C^2$ boundaries. In succeeding papers we will establish analogous $L^p$ results for the Cauchy-Leray integral on strongly $\mathbb{C}$-linearly convex domains with $C^{1,1}$ boundaries, and for the Szeg\H o projection, as well as certain holomorphic Cauchy-Fantappi\'e integrals, on strongly pseudoconvex domains with $C^2$ boundaries.\\

In considering each of these reproducing operators and their integral kernels, we may divide them into several classes. First, the \textquotedblleft universal\textquotedblright \ Bochner-Martinelli kernel,
which is unique among all these in that 
it is essentially independent of the domain. However its serious drawback is that it does not produce holomorphic functions in general
 (because it is not holomorphic in the parameter) 
 and so plays no role below. Next there are the \textquotedblleft canonical\textquotedblright 
  kernels that are determined by the domain in question, but do not depend on any particular choice of defining function for the domain. Among these are the Cauchy-Leray,
Szeg\H o, and Bergman 
kernels. Finally there are various kinds of non-canonical 
kernels, such as the kernels of the operators $B_\epsilon$ used below, that also depend on the choice of defining function for the domain.\\

Regularity properties of the Bergman projection and in particular its $L^p$ boundedness have been the object of considerable interest for more that 40 years. When the boundary of the domain $D$ is sufficiently smooth, decisive results were obtained in the following settings:
\begin{enumerate}
\item[(a)] 
When $D$ is strongly pseudoconvex, \cite{F}, \cite{PS}.
\item[(b)]When $D \subset \mathbb{C}^2$ and its boundary is of finite type \cite{M}, \cite{NRSW}.
\item[(c)]
When $D \subset \mathbb{C}^n$ is convex and its boundary is of finite type \cite{M1}, \cite{MS}.
\item[(d)] 
When $D \subset \mathbb{C}^n$  is of finite type and 
its Levi form is diagonalizable \cite{CD}.
\end{enumerate}

\noindent Related results were obtained in \cite{B}, \cite{B1}, \cite{BaLa},  \cite{BB},
\cite{BoLo}, \cite{EL}, \cite{H}, \cite{KP} and \cite{Z}.\\

Turning to the problem of what might be the suitable condition of minimal smoothness for a domain $D$, we see a clear distinction between the case of $\mathbb{C}$ (one complex variable) and $\mathbb{C}^n , \ n \geq 2$. In the former case, all the Cauchy-Fantappi\' e kernels reduce to the familiar Cauchy kernel
 (which is, of course, holomorphic in the parameter)
  and the natural limit of regularity of the boundary involves one derivative, if only because of the necessity, for the Cauchy integral and the Szeg\H o projection, that the 
  boundary 
  be rectifiable. Indeed deep results of this kind  
 -- \textquotedblleft near $C^1$\textquotedblright \ e.g., Lipschitz boundaries
 -- have been obtained by a number of mathematicians for the Cauchy integral and the Szeg\H o and Bergman projections. 
 (Recent results are
  in \cite{LS}, along with 
  a review of 
  the extensive literature.) However when $n \geq 2$, pseudoconvexity must intervene, and because of this the natural lower limit of regularity should by necessity involve two orders of differentiability. Now if we want the generality of considering all domains that are of class $C^2$, we then have essentially the following choice: either allow all (weakly) pseudoconvex domains, or restrict attention to strongly pseudoconvex domains. This is because to consider an intermediate class of domains (say those of \textquotedblleft finite-type\textquotedblright) \ would in effect require differentiability of higher order, related to the type in question. With these limitations in mind, we state our main result.\\

 \textbf{Theorem.}  \textit{Let} $D$ \textit{be a bounded domain in} $\mathbb{C}^n$ \textit{which has a} $C^2$ \textit{boundary that is strongly pseudoconvex.} \textit{Then the Bergman projection} $B$ \textit{of} $L^2(D)$ \textit{to the Bergman Space} $\vartheta L^2(D)$ \textit{is bounded on} $L^p(D)$, \textit{for} $1 < p < \infty$.\\

The main difference when dealing with the situation when $D$ is only of class $C^2$ compared with the cases of the more regular domains treated in (a) - (d) is that in each of these cases
 known formulas for the Bergman kernel, or at least size estimates, 
played a decisive role. In our general situation, such estimates are unavailable and we must proceed by a different analysis.\\

Our starting point is the idea used in \cite{KS} to study the Szeg\H o projection which was adapted by \cite{L} and then \cite{R} for the Bergman projection, all these when the domains are sufficiently smooth. In terms of the Bergman operator $B$, this proceeded by constructing another (non-orthogonal) projection $B_1$ via the Cauchy-Fantappi\`{e} formalism, corrected by the solution of a $\overline{\partial}$-problem. Then a simple formula connected $B$ with $B_1$  and $B^\ast_1$, and the problem was resolved because $B_1 - B_1^\ast$ was \textquotedblleft smoothing\textquotedblright \ (or sufficiently small). In our general situation this regularity or smallness is not possible, but what works instead can be described imprecisely as follows: One constructs an appropriate family $\{B_\epsilon\}_{\epsilon>0}$ of non-orthogonal projections. While this family does not approximate $B$, (in fact the norms of the $B_\epsilon$ are in general unbounded as $\epsilon \to 0$), suitable truncations of the $B_\epsilon$ approximate in a specific sense the \textquotedblleft essential part\textquotedblright \ of $B$. This is made precise in Lemma \ref{L:B-B-star-dec}. \\

There are two further results worth mentioning that follow from our analysis. First, not only is the operator $B$ bounded on $L^p$, but so is the operator $| B |$ whose kernel is the absolute value of the Bergman kernel. This is in sharp contrast with the Cauchy-Leray integral and the Szeg\H o projection, because such operators are non-absolutely convergent singular integrals, and this makes their treatment more intricate than that of $B$. Secondly we have the following approximation property: The linear space of functions holomorphic in a neighborhood of $\overline{D}$ is dense in $\vartheta L^p (D)$, $1<p<\infty$.\\
\section{The Levi polynomial and its variants}\label{1:Levi}\vspace*{.1in}
\subsection{Preliminaries}\label{Ss:preliminaries}
$D$ is a bounded domain in $\mathbb{C}^n$ which is of class $C^2$ and is strongly pseudoconvex. Thus there is a defining function $\rho \in C^2(\mathbb{C}^n)$ for $D$, with $D = \{ \rho < 0 \}$ and $| \triangledown \rho | > 0$ on $bD$, with $\rho$ strictly pluri-subharmonic.
Define $P_w(z)$, the Levi polynomial at $w$
by
\begin{equation*}
 P_w (z) = \sum_j\limits \frac{\partial \rho (w)}{\partial w_j}(z_j - w_j ) + \frac{1}{2} \ \sum_{j,k}\limits  \frac{\partial^2 \rho(w)}{\partial w_j w_k}(z_j - w_j)   (z_k- w_k ),
\end{equation*} 
which is a quadratic (holomorphic) polynomial in $z$.
 A basic property of $P_w$ is that
\begin{equation}\label{E:second-order-Taylor}
   \rho(w) + 2\,\Re P_w (z) + L_w(z-w)
   = \mathcal{T}_w (z)
\end{equation} 
is the second-order Taylor expansion of $\rho (z)$ about $w$. (See e.g. 
\cite{K}.) Here 
\begin{equation*}
    L_w(z-w) = \sum_{j, k} \frac{\partial^2 \rho (w)}{\partial w_j \partial \overline{w}_k}\, (z_j - w_j) \left( \overline{z}_k - \overline{w}_k \right)
\end{equation*}
is the Levi form and 
$$L_w (z-w)\geq c\, | z - w |^2 , \quad c > 0,$$
 by the strict pluri-subharmonicity we assumed. So we can write the above as
\begin{eqnarray}\label{E:Taylor-with-Levi}
   2\,\Re(-\rho (w) - P_w (z))=\\
    -\rho (z) - \rho (w) + L_w (z - w) + \littleo(|z - w |^2 )\notag
   \end{eqnarray}
as $|z-w|\to 0$. Now to extend the above when $z - w$ is not small, we define the function $g(w, z)$ by
\begin{equation}\label{E:def-g}
   g(w, z) = - P_w (z) \chi + | z - w |^2 (1 - \chi) - \rho (w).
\end{equation} 
Here $\chi = \chi(| z - w |^2 )$ is a $C^\infty$ function which is 1 when $|z - w | \leq \nicefrac{\mu}{2}$ and vanishes when $| z - w | \geq \mu$. We take $\mu$ to be a small constant, fixed so that \eqref{E:Taylor-with-Levi} and the strict positivity of $L_w$ guarantee that we have 
\begin{equation} \label{E:g-bounds}
2\,\Re\, g(w, z)\,\ge\,
\Bigg\{\!\!
\begin{array}{lll}
 - \rho (w) - \rho (z) + c | w-z |^2 ,
 &\text{if}\  |w-z | \leq \mu \\
 \quad &\\
 \qquad\qquad\qquad\qquad\qquad\  c, &\text{if}\  | w-z | \geq \mu
\end{array}
\end{equation}
for some constant $c > 0$. 

Now $g$, and its variants $g_\epsilon$ defined below, will be our basic tools.

Note $g(w,z)$ is holomorphic (or $C^\infty$) in $z$, hence smooth in that variable, but is only continuous in $w$, since we have only assumed $\rho$ is of class $C^2$. Because of this we introduce for each $\epsilon > 0$ an $n\times n$ matrix 
$$\displaystyle \tau^\epsilon (w) = \big(\tau^\epsilon_{j, k}(w) \big)$$ 
of functions that are each smooth (precisely: $C^2$ in $\overline{D}$), so that 
\begin{equation*}
   \sup_{w \in \overline{D}}\, \left| \frac{\partial^2 \rho (w)}{\partial w_j \partial w_k} - \tau^\epsilon_{j,k} (w) \right| \leq \epsilon \quad \text{ for all } 
   \quad 1 \leq j, \, k \leq n.
\end{equation*}
We define accordingly
\begin{equation}\label{E:P-eps-def}
   P^\epsilon_w (z) = \sum_j\limits \frac{\partial \rho (w)}{\partial w_j} \left( z_j - w_j \right) + \frac{1}{2} \sum_{j,k}\limits \tau^\epsilon_{j,k} (w)(z_j - w_j )(z_k - w_k )
\end{equation}
and
\begin{equation}\label{E:def-g-eps}
   g_\epsilon (w,z) = - P^\epsilon_w (z) \chi + | z - w |^2 (1 - \chi ) - \rho (w).
\end{equation} 
Note that since 
\begin{equation}\label{E:g-g-eps-close}
|\, g(w, z) - g_\epsilon (w, z) |\ \leq\ c\, \epsilon\, |w-z|^2
\end{equation}
once we have chosen $\epsilon$ sufficiently small, we can assert that 

\begin{equation} \label{E:g-eps-bounds}
2\,\Re\, g_\epsilon (w, z)\,\ge\,
\Bigg\{\!\!
\begin{array}{lll}
 - \rho (w) - \rho (z) + c | w-z  |^2 ,
 &\text{if}\  |w-z  | \leq \mu \\
 \quad &\\
 \qquad\qquad\qquad\qquad\qquad\  c, &\text{if}\  | w-z  | \geq \mu
\end{array}
\end{equation}
(after having decreased the sizes of $\mu$ and $c$, if necessary).

\subsection{Size estimates of the $g_\epsilon$}\label{Ss:g-eps-size}
There are two properties of the $g_\epsilon$ that are needed below: size and symmetry estimates. The first is contained in the following.
We use the notation
$$\langle a, b \rangle=
\sum_j\limits a_j b_j,$$
where $a = (a_1 ,  \ldots,  a_n )$ and $b = (b_1 ,  \ldots,  b_n )$ are vectors in $\mathbb{C}^n$. 
\begin{proposition}\label{P:g-g-eps}
   If $z$ and $w$ are in $\overline{D}$ and $\epsilon$ is sufficiently small, then\\
\item[(a)] $|g(w, z) | \approx | \rho (w) | + |\rho (z) | +  | \Im \langle \partial \rho (w), \ w - z \rangle | + | w-z  |^2$\\
\item[(b)] $|g_\epsilon (w, z) | \approx | g(w, z) |$.
\end{proposition}
Here (and below) the constants implicit in the equivalences $\approx$ are independent of $\epsilon$.
\begin{proof}
   Note that $|\rho (z) | = -\rho (z) \text{ and} \ |\rho (0) | = -\rho (w)$, since $z$ and $w$ are in $\overline{D}$. Next when $|w-z| \geq \mu/ 2$, the assertion (a) is immediate from \eqref{E:g-bounds}. We turn to the case
    $|w-z | \leq \mu / 2$.\\

We first observe that \eqref{E:Taylor-with-Levi} and \eqref{E:def-g} show that 
\begin{equation}\label{E:Re-g-approx}
   | \Re\,g(w,z) | \approx | \rho (z) | + | \rho (w) | + | w-z |^2.
\end{equation} 
Also by \eqref{E:def-g}
 $$\Im g(w, z) = - \Im\, P_w (z)$$ 
and
by the definition of $P_w (z)$ this yields 
\begin{equation}\label{E:Im-g-ineq}
   \left|\, \Im\big( g(w, z) - \langle \partial \rho (w), w - z\rangle \big) \right|
    \leq c\, | z - w |^2.
\end{equation}
Hence
\begin{equation*}
  \left| \Im\, g(w, z) \right| + c | z - w |^2 \geq  \left| \Im  \langle \partial \rho (w), w - z \rangle \right| \, , 
  \end{equation*} 
which gives 
$$| g(w, z) |\ \gtrsim\ | \rho (z) | + | \rho (w) | + 
|\, \Im\langle \partial \rho (w), \ w - z \rangle | + | w-z |^2.$$

To see the reverse, note that by \eqref{E:Im-g-ineq}
$$
|\,\Im g(w, z) | \ \leq \ |\, \Im \langle \partial \rho (w), w - z \rangle | + c\, |w-z |^2,$$
 and combining this with \eqref{E:Re-g-approx} gives
$$| g(w, z) | \ \lesssim\  | \rho (z) | \ + \ | \rho (w) | \ + \ |\, \Im \langle \partial \rho (w), w - z \rangle | \ + \ | w - z |^2.$$
Hence (a) is fully established.
The conclusion (b) follows immediately from (a) and 
\eqref{E:g-g-eps-close}.
\end{proof}

From now on all our statements will be restricted to the $\epsilon$ that are small enough so that both 
\eqref{E:g-eps-bounds}  and the above conclusion holds.

\subsection{Symmetries of the $g_\epsilon$}\label{Ss:g-eps-symm}
One can observe that by Proposition \ref{P:g-g-eps} we have the symmetries
\begin{equation}\label{E:g-g-eps-symm}
   | g(w, z) | \approx | g(z, w) | \ \ \ \text{and} \ \ \  \ |g_\epsilon (w, z) | \approx | g_\epsilon (z, w) |.
\end{equation} 
In fact since $\rho$ is of class $C^2$,
\begin{equation*}
   \langle \partial \rho (w) , z - w \rangle - \langle \partial \rho (z), z - w \rangle = \bigO(| w-z |^2) , 
\end{equation*}
thus
\begin{equation*}|\,\Im \langle \partial \rho (w), w-z  \rangle | + | w-z |^2 \approx |\, \Im \langle \partial \rho (z), w - z \rangle | + | w-z |^2
\end{equation*}
proving by (a) of Proposition \ref{P:g-g-eps}, the asserted symmetry of 
$| g(w, z) |$ and hence of $| g_\epsilon (w, z) |$.\\

However for what follows below a more refined version of this symmetry is crucial.
\begin{proposition}\label{P:g-eps-g-eps-conj}
  For any $\epsilon > 0$, there is $\delta = \delta_\epsilon > 0$, so that 
\begin{equation*}
   | g_\epsilon (w, z) - \overline{g}_\epsilon (z, w) | \ \leq \ 
   c\, \epsilon\, | z - w |^2 , \quad \text{ if }\quad  | z - w | < \delta_\epsilon
\end{equation*}
\end{proposition}
The inequality above takes into account the modules of continuity of the second derivatives of $\rho$. In this connection we define 
\begin{equation*}
  \omega_{j, k} ( \delta ) = \sup_{| z - w | \leq \delta}\limits \left( \left| \frac{\partial^2 \rho (z)}{\partial z_j \partial z_k} - \frac{\partial^2 \rho (w)}{\partial w_j \partial w_k} \right| + \left|    \frac{\partial^2 \rho (z)}{\partial {z}_j \partial \overline{z}_k} - \frac{\partial^2 \rho (w)}{\partial w_j \partial \overline{w}_k} \right| \right)
\end{equation*}
taken over all $z, w \in \overline{D}$. We set 
$$\displaystyle \omega(\delta ) = \sum_{j, k} \omega_{j, k} (\delta )$$

Then in view of the uniform continuity of the second derivatives of $\rho$, the function $\omega (\delta )$ decreases to 0 as $\delta \to 0$.\\

We observe next, with $\mathcal{T}_w (z)$ denoting the Taylor polynomial of order 2 of $\rho(z)$ centered at $w$, see \eqref{E:second-order-Taylor}, that then
\begin{equation}\label{E:ineq-Taylor-pol}
   \left| \rho (z) - \mathcal{T}_w (z) \right|\ \leq\ c\, \omega ( | z - w | ) | z - w |^2,\quad if \quad |w-z|<\delta.
\end{equation} 
In fact, for any $C^2$ function $f$ on $\mathbb{R}$ we have the identity
\begin{equation*}
f(t) - f(0) - t\, f^\prime (0) - \frac{t^2}{2}\, f^{\prime\prime}(0)\ =\
 \int^t_0\limits (t-s)\left( f^{\prime\prime}(s) -  f^{\prime\prime}(0) \right) ds;
\end{equation*}
also the integral is majorized by
\begin{equation*}
   \frac{1}{2}\, t^2 \sup_{0 \leq s \leq t}
   \left(t\left| f^{\prime\prime}(s) - f^{\prime\prime}(0) \right|\right) . 
\end{equation*}
Applying this to $f(t) = \rho(w + t(z - w))$, then yields the inequality 
\eqref{E:ineq-Taylor-pol}, when we take $t=1$.

To prove Proposition \ref{P:g-eps-g-eps-conj}, note by 
\eqref{E:def-g-eps} that if we take $\delta \leq \mu/2$, so that $| z - w | \leq \mu/2$, by \eqref{E:P-eps-def} we are reduced to considering 
\begin{eqnarray}\label{E:temp-1}
   \rho (z) - \rho (w) + \langle \partial \rho (w), w-z \rangle + \langle \overline{\partial} \rho (z), \overline{w} - \overline{z} \rangle +\\
   + \frac{1}{2} \langle\, \overline{\tau}^{\,\epsilon} (z); \ \overline{w} - \overline{z}, \ \overline{w} - \overline{z} \rangle - \frac{1}{2} \langle \tau^\epsilon (w); \ w - z, \ w - z \rangle .\notag
\end{eqnarray} 
Here we are using the short hand
\begin{equation*}
 \langle M; a, b  \rangle = \sum\limits_{j, k} M_{j,k}\, a_j\, b_k
\end{equation*}
if $M = (M_{j,k})$ is an $n \times n$ matrix and $a = (a_1, \ldots, a_n ),
 \   b = (b_1, \ldots, b_n )$ are vectors.
Next we write the Taylor expansion of order 1 of 
$\partial \rho (w)/\partial w_j$ centered at $z$, and use the same argument that led to \eqref{E:ineq-Taylor-pol}. We thus see that 
\begin{equation*}
   \frac{\partial \rho (w)}{\partial w_j} = \frac{\partial \rho (z)}{\partial z_j} + \sum_k\limits \left( \frac{\partial^{2} \rho (z)}{\partial z_j \partial z_k} (w_k - z_k )   + \frac{\partial^2 \rho (z)}{\partial z_j \partial \overline{z}_k} (\overline{w}_k - \overline{z}_k ) \right) + R_1 ,
\end{equation*}
where $\displaystyle | R_1 | \ \leq\ c\, \omega ( | w - z | ) | w - z |$.
If we insert this for $\partial \rho (w)$ in \eqref{E:temp-1} we obtain
\begin{eqnarray*}
 g_\epsilon (w,z) - \overline{g}_\epsilon (z, w) =\\
  \rho (z) - \rho (w) + \langle \partial\rho (z), w - z \rangle 
  + \langle \overline{\partial} \rho (z), \ \overline{w} -  \overline{z} \rangle +\\
 + \frac{1}{2} \left \langle \frac{\partial^2 \rho (z)}{\partial z \partial z}; \ w - z, w - z \right \rangle + \frac{1}{2} \left \langle \frac{\partial^2 \rho (z)}
 {\partial \overline{z} \partial \overline{z}} ; \ \overline{w} - \overline{z} , \ \overline{w} - \overline{z} \right \rangle +\\
 +  \left \langle \frac{\partial^2 \rho (z)}{\partial z \partial \overline{z}} ; \ w - z, \overline{w} - \overline{z} \right \rangle 
+ \bigO(| w - z | R_1 ) + R_2 ,
\end{eqnarray*}
\noindent where, 
\begin{align*}
   R_2 = &\frac{1}{2} \left \langle \frac{\partial^2 \rho (z)}{\partial z \partial z} - \tau^\epsilon (w); \ w - z, \ w - z \right \rangle\\
&+ \frac{1}{2} \left \langle \overline{\tau}^{\,\epsilon} (z) - \frac{\partial \rho(z)}{\partial \overline{z} \partial \overline{z}} ; \ \overline{w} - \overline{z} , \   \overline{w} - \overline{z} \right \rangle .
\end{align*}
So this can be written as
\begin{equation*}
   g_\epsilon (w, z) - \overline{g}_\epsilon (z, w) = \rho (z) - 
   \mathcal{T}_w (z) + \bigO(| w - z | R_1 ) + R_2 .
\end{equation*}
Now recall \eqref{E:ineq-Taylor-pol}, and by the same token, $| R_1 | \leq c \omega (| z - w | ) | z - w |$, while clearly $| R_2 | \leq c  \epsilon | z - w |^2 
 $. Thus we get the inequality $| g_\epsilon (w, z) - \overline{g}_\epsilon (z, w) | \leq c \epsilon | z - w |^2$ for $ | z - w | \leq \delta$, as soon as we choose $\delta \leq \delta_\epsilon$, with $\omega (\delta_\epsilon ) \leq \epsilon$. Proposition \ref{P:g-eps-g-eps-conj} is therefore proved.\\

We should note that in this section we have not used the $C^2$ regularity of the $\left( \tau^\epsilon_{j, k} (w) \right)$; this will be needed in the next section.

\section{An \textquotedblleft approximation\textquotedblright \ for the Bergman projection}\label{S:approx}
The first step in the study of the Bergman projection $B$ for the domain $D$ is to form a family $\left\{ B_\epsilon \right\}_{\epsilon>0}$ of operators that are each (non-orthogonal) projections from $L^2(D)$ to the Bergman space
$\vartheta L^2(D)$, that is
the subspace of $L^2(D)$ consisting of holomorphic functions. While these operators do not converge to the Bergman projection $B$ as $\epsilon \to 0$ (in fact their norms are in general unbounded as $\epsilon \to 0$), they nevertheless play a crucial role in the analysis of $B$.
The operator $B_\epsilon$
will be given as a sum
\begin{equation*}
   B_\epsilon = B^1_\epsilon + B^2_\epsilon .
\end{equation*}
The first, $B^1_\epsilon$, is an adaptation of a Cauchy-Fantappi\'{e} integral constructed explicitly using the function $g_\epsilon (w, z)$ of the previous section. The second, $B^2_\epsilon$, is a correction obtained by solving a $\overline{\partial}$-problem in a domain that strictly contains $D$. Here we follow the approach of \cite{KS}, \cite{L}, and \cite{R}, which as it stands, works only in the case of smooth domains (say of class $C^3$). Another significant difference is that now we need a family
 $\{B_\epsilon\}_{\epsilon>0}$, as opposed to a single such operator, and that the properties of $B_\epsilon$ as $\epsilon \to 0$ are now indispensable.
We turn first to $B^1_\epsilon$.
\subsection{The Cauchy-Fantappi\'{e} part}
Keeping in mind the $\left\{ g_\epsilon (w, z) \right\}_{\epsilon>0}$ given by 
\eqref{E:def-g-eps}, we define the (1, 0)-forms $\eta_\epsilon$ by
\begin{align}\label{E:def-eta-eps}
   \eta_\epsilon (w, z) = \chi &\left( \sum_j \frac{\partial \rho (w)}{\partial w_j} dw_j - \frac{1}{2} \sum_{j, k} \tau^\epsilon_{j, k} (w) (w_k - z_k ) dw_j \right)\\
& + (1 - \chi ) \sum_j \left( \overline{w}_j - \overline{z}_j \right) dw_j ,\notag
\end{align} 
so that
\begin{equation}\label{E:re-def-g-eps}
   g_\epsilon (w, z) = \langle \eta_\epsilon (w, z), w - z \rangle - \rho (w).
\end{equation} 
\indent Now with the above, the definition of $B^1_\epsilon$ will be taken to be
\begin{equation}\label{E:def-B-1-eps}
   B_\epsilon^1 (f)(z) = \frac{1}{(2 \pi i)^n} \int_D
    f(w)(\overline{\partial}_w G_\epsilon )^n (w, z),\quad z\in D
\end{equation} 
where we have set
\begin{equation}\label{E:def-G-eps}
G_\epsilon (w, z) \ =\  \frac{\,\eta_\epsilon (w, z)\,}{\,g_\epsilon (w, z)\,}
\end{equation}

From inequality \eqref{E:g-eps-bounds} it follows that
for any $z_0 \in D$ there is an open ball $\mathcal{B}_{z_0}$ centered at $z_0$, and a neighborhood $\mathcal{N}_{z_0}$ of $bD$, so that
 $$g_\epsilon (w, z) + \rho (w) \neq 0,\quad \text{if }\quad  z \in \mathcal{B}_{z_0}, \ \ w \in \mathcal{N}_{z_0}.$$

Thus if  we set
$$\displaystyle\widehat{G}_\epsilon (w, z)\  =\  
\frac{\eta_\epsilon (w, z)}{g_\epsilon (w, z) + \rho (w)}\ =\ 
\frac{\eta_\epsilon (w, z)}{\langle\eta_\epsilon (w, z), w-z\rangle}$$ 
then
\begin{equation}\label{E:generating-G}
   \langle \widehat{G}_\epsilon (w, z), w - z \rangle = 1 , \ \text{ for } \ z \in \mathcal{B}_{z_0} , \ w \in \mathcal{N}_{z_0} .
\end{equation} 
A $(1, 0)$-form $\mathcal G(w, z)$ 
is a 
{\it generating form over $z_0$} if it is of class $C^2$ in the variable $w$ and satisfies \eqref{E:generating-G}, and it is a
{\it generating form over $D$} if it is generating over $z$ for any $z\in D$.
One has the following
\begin{lemma}\label{L:generating-G}
 Suppose that $\mathcal G$ is a generating form over $z\in D$.
 Then
   \begin{equation}\label{E:reproducing-G}
      f(z) = \frac{1}{(2 \pi i)^n} \int_{bD}\limits f(w) 
      \left(\mathcal G \wedge \left( \overline{\partial}_w 
      \mathcal G \right)^{n-1}\right)\!(w, z)
   \end{equation} 
for any $f$ that is holomorphic in $D$ and continuous in $\overline{D}$.
\end{lemma}
\noindent Here we have set 
$\left( \overline{\partial}_w 
\mathcal G \right)^{n-1}\!\! = \overline{\partial}_w \mathcal G  \wedge \overline{\partial}_w \mathcal G \ldots \wedge \overline{\partial}_w \mathcal G,$ 
with
the wedge product taken $n-1$ times. The expression
$$
\frac{1}{(2\pi i)^n}\, \mathcal G \wedge \left( \overline{\partial}_w \mathcal G \right)^{n-1}
$$
is called the {\it Cauchy-Fantappi\'e form of order 0 generated by 
$\mathcal G$}.


Starting with \eqref{E:reproducing-G} we might hope to use Stokes' theorem to show that 
\begin{equation}\label{E:representing-B_1-eps}
\displaystyle B^1_\epsilon (f)(z) = \frac{1}{(2\pi i)^n} \int_{bD}\limits f(w)
\left(G_\epsilon \wedge \left( \overline{\partial}_w G_\epsilon \right)^{n-1}\right)\!(w, z),\quad z\in D
\end{equation}
 and hence that 
 \begin{equation}\label{E:reproducing-B-1-eps}
 f(z)\ =\ B^1_\epsilon (f)(z),\quad z\in D
 \end{equation} at least for holomorphic $f$ that are of class $C^1(\overline{D})$.
The problem is that $G_\epsilon$ is not a generating form
(it is not of class $C^2(\overline{D})$, and it does not satisfy
\eqref{E:generating-G}),
 nor do we want to restrict ourselves to the $f$ that are in $C^1 (\overline{D})$, but rather to the holomorphic $f$ that are in (say) $L^1 (D)$. We get around these obstacles in two stages. First, establishing 
 \eqref{E:reproducing-B-1-eps}
 for $f$ in $C^1(\overline{D})$, and then relaxing this requirement on $f$. Our result is as follows
\begin{proposition}\label{P:reproducing-B-1-eps}
   Suppose $f$ is holomorphic in $D$ and belongs to $L^1 (D)$. 
   Then, the identity \eqref{E:reproducing-B-1-eps}
holds in $D$, with $B^1_\epsilon$ given by 
\eqref{E:def-B-1-eps}.
\end{proposition}
The first step is to approximate the defining function $\rho$ (which is of class $C^2$) by a sequence $\{ \rho^r \}$ of functions each of class $C^3$, so that the $\rho^r$ and their derivatives of order $\leq 2$ converge uniformly to $\rho$. Recalling the definition of $\eta_\epsilon$ in \eqref{E:def-eta-eps} we set 
\begin{align}\label{E:def-eta-r-eps}
   \eta^r_\epsilon (w,z) = \  & 
   \chi \left( \partial \rho^r (w) - \frac{1}{2} \sum_{j, k} \tau^\epsilon_{j, k} (w) (w_k - z_k )\, dw_j \right)\\
& 
+ (1 - \chi ) \sum_{j}\left( \overline{w}_j - \overline{z}_j \right) dw_j ,\qquad 
\text{and}\notag
\end{align}

\begin{equation*}
g^r_\epsilon (w, z) = \langle \eta^r_\epsilon (w, z), w - z \rangle - \rho (w)
\end{equation*}
\medskip

 We note that for $z \in D$, the $\eta^r_\epsilon (w, z) \to \eta_\epsilon (w, z)$ and $g^r_\epsilon (w, z) \to g_\epsilon (w, z)$ as $r \to \infty$, uniformly for $w \in \overline{D}$ together with their first derivatives in $w$. 
 Moreover by \eqref{E:g-eps-bounds},
 we have $g^r_\epsilon (w, z) \neq 0$ when $w \in bD$, for sufficiently large $r$. We then set 
 \begin{equation}\label{E:def-G-r-eps}
 G^r_\epsilon(w, z) = \frac{\,\eta^r_\epsilon (w, z)\,}{g^r_\epsilon (w, z)} \quad \text{and}\quad 
 G_\epsilon(w, z) = \frac{\,\eta_\epsilon (w, z)\,}{g_\epsilon (w, z)}
 \end{equation}
 and we see that the (1,0)-forms $G^r_\epsilon$ converge to $G_\epsilon$
  uniformly in $\overline{D}$, together with their first derivatives. Finally each $G^r_\epsilon$ has continuous derivatives of order $\leq 2$ for $w \in \overline{D}$, (because $\rho^r$ was of class $C^3$ and $\tau^\epsilon$ of class $C^2$).

Now assume $f$ is holomorphic in $D$ and belongs to $C^1 (\overline{D})$. Then $\overline{\partial}_w \left( f \,G^r_\epsilon \wedge \left(  \overline{\partial}_w G^r_\epsilon \right)^{n-1} \right) = f \,\overline{\partial}_w \left( G^r_\epsilon \wedge \left(  \overline{\partial}_w G^r_\epsilon \right)^{n-1} \right) = f\, \left( \overline{\partial}_w G^r_\epsilon \right)^n$, since $\overline{\partial}_w \left( \overline{\partial}_w G^r_\epsilon \right)^{n-1} =  0$. And since $f G^r_\epsilon \wedge \left( \overline{\partial}_w G^r_\epsilon \right)^{n-1}$ is an $(n, n-1)$-form, 
$$\overline{\partial}_w \left( f\, G^r_\epsilon \wedge \left( \overline{\partial}_w G_\epsilon^r \right)^{n-1} \right)\ =\ d_w \left( f\, G_\epsilon^r \wedge \left( \overline{\partial}_w G_\epsilon^r \right)^{n-1} \right).$$

\noindent So Stokes' theorem gives
\begin{equation}\label{E:Stokes}
   \int_D f(w) \left( \overline{\partial}_w G^r_\epsilon \right)^n = \int_{bD} f(w) \,G^r_\epsilon \wedge \left(\overline{\partial}_w G^r_\epsilon \right)^{n-1}.
\end{equation} 
One notes that when $w \in bD$, we have $G^r_\epsilon = \widehat{G}^r_\epsilon$ where we have defined
\begin{equation}\label{E:def-G-hat-eps-r}
\widehat{G}^r_\epsilon (w, z)\ =\ 
\frac{\eta^r_\epsilon (w, z)}{g^r_\epsilon (w, z) + \rho (w)}\ =\ 
\frac{\eta^r_\epsilon (w, z)}{\langle\eta^r_\epsilon (w, z), w-z\rangle}
\end{equation}
Also one checks that 
$$
G^r_\epsilon \wedge \left( \overline{\partial}_w G_\epsilon^r \right)^{n-1}=\ \widehat{G}_\epsilon^r \wedge \left( \overline{\partial}_w \widehat{G}_\epsilon^r \right)^{n-1}\quad \text{when}\quad  w \in bD.
$$

However now, $\widehat{G}^r_\epsilon$ is a generating form over $D$
(it is of class $C^2(D)$ in the variable $w$, and it satisfies 
\eqref{E:generating-G} for each $z\in D$),
 thus by Lemma \ref{L:generating-G} and \eqref{E:Stokes}, 
\begin{equation*}
   f(z) = \frac{1}{(2 \pi i)^n} \int_D\limits f(w) \left( \overline{\partial}_w G_\epsilon^r \right)^n ,\quad z\in D
\end{equation*}
and a passage to the limit, $r \to \infty$, then yields the representation 
\eqref{E:reproducing-B-1-eps}, under the restriction that $f$ is holomorphic and is  in $C^1(\overline{D})$.

To relax this requirement on $f$ we use what we have established, but instead for domains that are interior to our domain $D$, proceeding as follows:
for any positive $\lambda$, define $D_\lambda$ by 
\begin{equation*}
D_\lambda = \{ w: \rho (w) + \lambda < 0 \}.
\end{equation*}
Thus $D_\lambda$ arises by replacing the defining function $\rho$ by $\rho_\lambda = \rho + \lambda$. Note that $\overline{D}_\lambda \subset D_0$ and more generally $\overline{D}_{\lambda^\prime} \subset D_\lambda$ \ if \ $\lambda^\prime > \lambda$.

Returning to $\displaystyle g_\epsilon$ and $\eta_\epsilon$ if we define 
\begin{equation}\label{E:def-G-tilde-eps-lambda}
\widetilde{G}_\epsilon^\lambda (w, z)\ =\ 
\frac{\,\eta_\epsilon (w, z)\,}{g_\epsilon (w, z) - \lambda}
\end{equation}
  then whenever $f$ is holomorphic in $D$, we have by what has already been proved,
\begin{equation*}
   f(z) = \frac{1}{{(2 \pi i)}^n} \int_{D_\lambda}\limits f(w) \left( \overline{\partial}_w \widetilde{G}_e^\lambda \right)^{n-1}   \ , \ \text{ whenever } z \in D_\lambda ,
\end{equation*} 
because $f$ is automatically in $C^1(\overline{D}_\lambda )$.\\

However it is also clear that for fixed $z \in D$ and for fixed $\lambda(z)$ chosen so that $z\in D_\lambda$ for all 
$0<\lambda<\lambda (z)$, that 
\begin{equation*}
   \widetilde{G}^\lambda_\epsilon \to G_\epsilon , \ \text{as } \lambda \to 0 \ , \text{ uniformly for  } w \in D_{\lambda (z)} \ ,
\end{equation*}
with a corresponding convergence of derivatives of order one.
However if $z \in D$, 
\begin{align*}
   (2 \pi i)^n f(z) = &\int_{D_\lambda}\limits f(w) \left( \overline{\partial}_w \widetilde{G}_\epsilon^\lambda \right)^n = \int_D\limits f(w) \left( \overline{\partial}_w G_\epsilon \right)^n\\
&+ \int_{D_\lambda} f(w) \left[ \left( \overline{\partial}_w \widetilde{G}_\epsilon^\lambda \right)^n - \left( \overline{\partial}_w G_\epsilon \right)^n \right] + \int _{D\setminus D_\lambda}\limits f \left( \overline{\partial}_w G_\epsilon \right)^n
\end{align*}
Now the second integral on the right tends to zero because of the nature of the convergence of $\overline{\partial}_w \widetilde{G}_\epsilon^\lambda$ to $\overline{\partial}_w G_\epsilon$, as $\lambda \to 0$; and the third integral tends to zero in view of the assumption that $f$ is integrable on $D$. Proposition 4 is therefore proved.
\subsection{The correction}\label{Ss:correction}
While the operator $B^1_\epsilon$ satisfies the reproducing property 
\eqref{E:reproducing-B-1-eps},
the function $B^1_\epsilon (f)$ is not necessarily holomorphic for general $f$ that are not holomorphic. This is because $G_\epsilon (w, z) = \eta_\epsilon (w, z) /g_\epsilon (w, z)$ is holomorphic in $z$ only for $z$ near $w$ 
(see \eqref{E:def-eta-eps} and \eqref{E:re-def-g-eps}). We now proceed to correct $B^1_\epsilon$ to overcome this difficulty. We write
\begin{equation}\label{E:def-ker-B-1-eps}
    B^1_\epsilon (w, z) = \left( \overline{\partial}_w G_\epsilon \right)^n(w, z) ,
\end{equation}
so that
\begin{equation*}
  B^1_\epsilon (f)(z) = \frac{1}{(2 \pi i)^n} \int_D f(w)\, B^1_\epsilon (w, z).
\end{equation*}
\begin{proposition}\label{P:correction}
   There is an $(n,n)$ form (in $w),  B^2_\epsilon (w, z)$, depending on $z$, that is continuous for $(w, z) \in \overline{D} \times \overline{D}$, so that if 
\begin{equation*}
   B_\epsilon (f) (z) = \frac{1}{(2 \pi i)^n} \int_D f(w)\! \left( B^1_\epsilon (w, z) + B^2_\epsilon (w, z) \right), \text{ then }
\end{equation*}
\begin{itemize}
\item[(a)] $B_\epsilon (f)(z)$ is holomorphic for $z \in D$, for each $f \in L^1 (D)$.
\label{C2}\item[(b)] In addition, if $f$ is also holomorphic in $D$ then $B_\epsilon (f)(z) = f(z)$, \ for  \ $z \ \in D$.
\end{itemize}
\end{proposition}
The proof proceeds by solving a $\overline{\partial}$-problem for a domain $\Omega$, strictly containing $D$, which has a smooth $(C^\infty )$ boundary and is strongly pseudoconvex. Here the focus will be on the $z$-variable, with $w\in\overline D$ fixed, (a reversal of the attention paid to $w$ and $z$ above).

We define the \textquotedblleft parabolic\textquotedblright \ region 
\begin{equation*}
  \mathcal{P}_w = \left\{ z: \rho (z) + \rho (w) < \frac{c}{2} | z - w |^2 \right\},\quad w\in\overline D.
\end{equation*}
Here $c$ is the constant appearing in 
\eqref{E:g-bounds}. We note that $\mathcal{P}_w \supset D$, and in fact $\mathcal{P}_w \supset \overline{D}$ for $ w \in D$; moreover when $w \in bD$, the exterior of $\mathcal{P}_w$ intersects $\overline{D}$ only at $w$. We also set $\mathcal{B}_w$ to be the open ball centered at $w, \left\{ z: | z - w | < \mu /2 \right\}$ where $\mu$ is the constant which occurs in the inequality 
\eqref{E:g-bounds}. We will use the notation $D_\lambda = \left\{ z: \rho(z) < - \lambda \right\}$ which appeared above, but now for $\lambda$ negative, with $\lambda = - \lambda_0$, and $\lambda_0 > 0$. We have 
\begin{equation}\label{E:union}
  \mathcal{P}_w \cup \mathcal{B}_w \supset D_{- \lambda_0}, \quad \text{ where}\quad  \lambda_0 = c\mu^2/8 .
\end{equation} 

The proof of \eqref{E:union}  divides into two cases: when
 $| z-w | \geq \mu /2$ and when $| z - w | <  \mu /2$. In the first case, if $z \in D_{- \lambda_0}$ then $\rho   (z) < \lambda_0$, and
 since $\rho (w) \leq 0$, it follows from our choice of $\lambda_0$ that $z \in \mathcal{P}_w$. In the second case, $z$ is automatically in $\mathcal{B}_w$, and hence  \eqref{E:union} is proved. It is equally easy to see that we also have
 \begin{equation}\label{E:intersection}
   \mathcal{P}_w \cap \mathcal{B}_w \neq \emptyset ,\quad 
   w\in \overline{D}.
 \end{equation}

Next, by approximating $\rho$, the defining function of $D$, by an appropriate $C^\infty$ function which is close to $\rho$ and its derivatives of order not exceeding two, we can find $\widetilde{\rho}$, so that $\rho + \lambda_0 /2 < \widetilde{\rho} < \rho + \lambda_0$, and so that the domain $\Omega = \{ z: \widetilde{\rho}(z) < 0 \}$ has a boundary that is $C^\infty$ and strongly pseudoconvex. Note that 
\begin{equation}\label{E:Omega-D-lambda}
\overline D\subset D_{- \lambda_0/2} \subset \Omega \subset D_{- \lambda_0}
\end{equation}
and so
we have
\begin{equation*}
   \overline{D} \subset \Omega , \ \text{ and } \ \overline{\Omega} \subset \mathcal{P}_w \cup \mathcal{B}_w , \ \text{ for every } w \in \overline{D} .
\end{equation*} 
\indent We now set up our $\overline{\partial}$-problem on $\Omega$. For each $w \in \overline{D}$ we denote by $F(w, z) = F_\epsilon (w, z)$ the following double form, which is of type $(0, 1)$ in $z$, and of type $(n,n)$ 
in $w$ 
\begin{equation} \label{E:def-d-bar-data}
F(w, z)\ =\
\left\{
\begin{array}{rcl}
- \overline{\partial}_z \left( B^1_\epsilon (w, z) \right),
  &\text{if}&z \in \mathcal{P}_w\\
 \quad &\\
 \quad\  0, &\text{if}&  z\in \mathcal B_w
 \end{array}
 \right.
\end{equation}
In fact, note by \eqref{E:def-eta-eps} and \eqref{E:re-def-g-eps}, that $g_\epsilon (w, z) , \eta_\epsilon (w, z)$
 are holomorphic in $z$ for $z \in \mathcal{B}_w$ and by \eqref{E:g-eps-bounds}, $g_\epsilon (w, z)$ is non-vanishing in $\mathcal P_w\cap \mathcal B_w$.
%
Thus $\overline{\partial}_z \left( B^1_\epsilon (w, z) \right) = 0$ in $\mathcal P_w\cap \mathcal{B}_w$ and so $F(w, z)$ is defined consistently in 
$\mathcal{P}_w \cup \mathcal{B}_w$. It is also clear from this and from
\eqref{E:intersection} that $F(z, w)$ is $C^\infty$ for $z \in \mathcal{P}_w \cup \mathcal{B}_w$, and as such it depends continuously on $w \in \overline{D}$. 
Moreover $\overline{\partial}_z F(w, z) = 0$, 
 for $z \in \mathcal{P}_w \cup \mathcal{B}_w , \ w \in \overline{D}$.\\

Now let $S = S_z$ be the solving operator, giving the normal solution of the problem $\overline{\partial} u = \alpha$ in $\Omega$, via the $\overline{\partial}$-Neumann problem, so that $u = S(\alpha)$ satisfies the above whenever $\alpha$ is a $(0, 1)$-form with $\overline{\partial} \alpha = 0$. We set
 \begin{equation}\label{E:def-correction}
 B^2_\epsilon (w, z) = S_z (F(w, \cdot) ).
 \end{equation}
  Then by the regularity properties of $S$ 
for which (see e.g. \cite[chapters 4 and 5]{CS}, or \cite{FK}), we have that $B^2_\epsilon (w, z)$ is $C^\infty (\overline{\Omega})$, as a function of $z$, and continuous for $w \in \overline{\Omega}$. Moreover 
$\overline{\partial}_z \left( B^2_\epsilon (w, z ) \right) 
= -\overline\partial_zB^1_\epsilon (w, z)$,
 for $z\in D$, so $\overline{\partial}_z \left( B_\epsilon (w, z) \right) = 0$ for $ z \in D$. 
So conclusion
(a) is proved. To establish conclusion (b) it suffices to see that 
\begin{equation}\label{E:C2}
   \int_D f(w)F(w, z) = 0 \quad \text{ for } \quad z \in \overline{\Omega} ,
\end{equation} 
whenever $f \in 
\vartheta
L^1(D)$.
In fact if \eqref{E:C2} holds, then 
$$
\int_D f(w)B^2_\epsilon (w,z)= 
 \int_D f(w)S_z (F (w, z)) = 
S_z\left(\int_D f(w)F(w, z)\right)=0
$$
thus (b) is a consequence of \eqref{E:C2}.
Note that we have 
$$\int_Df(w) F(w,z) = - \int_D  \!f(w)\,\overline{\partial}_z B^1_\epsilon (w,z) = 0\quad \text{for}\quad  z \in D,
$$ but this does not suffice to give 
\eqref{E:C2} (for $z \in \overline{\Omega}$). To prove 
\eqref{E:C2} we need to use Stokes' theorem, and as in the proof of 
Proposition \ref{P:reproducing-B-1-eps}, we assume initially that $f \in C^1 (\overline{D})$, besides being holomorphic. \\

Now set $\displaystyle G^r_\epsilon = \eta^r_\epsilon /g^r_\epsilon$ as in the proof of Proposition \ref{P:reproducing-B-1-eps}, with $g^r_\epsilon (w,z) = \langle \eta^r_\epsilon (w, z), w - z \rangle - \rho(w)$, see
 \eqref{E:def-eta-r-eps} -- \eqref{E:def-G-r-eps}.
Also, set 
\begin{equation*}
F^r(w, z)\ =\
\left\{
\begin{array}{rcl}
- \overline{\partial}_z\, \overline{\partial}_w\! \left({G}^r_\epsilon \wedge \left( \overline{\partial}_w {G}^r_\epsilon \right)^{n-1}\right)\!,&\text{if}& z \in \mathcal{P}_w\\
 0, &\text{if}&  z\in \mathcal B_w.
 \end{array}
 \right.
\end{equation*}
Now
\begin{equation*}
   \int_D f(w)F^r (w,z)= 
   \int\limits_{D\setminus\mathcal{B}_w}\!\! f(w) F^r(w,z)=
   \int\limits_{D\setminus\mathcal{B}^\prime_z}\!\! f(w)F^r(w,z)
\end{equation*}
where $\mathcal{B}^\prime_z = \{ w: | w - z | < \mu /2 \}$, \ \ since \ \ $w \in \mathcal{B}^\prime_z$ \ \ exactly when $z \in \mathcal{B}_w$.\\

\noindent However the boundary of $D\setminus\mathcal{B}^\prime_z$ consists of two parts, denoted $I$ and $II$ below, where 
$$I=bD \cap (\mathcal{B}^\prime_z )^c,\qquad \text{and}\qquad II=D \cap b \mathcal{B}^\prime_z.$$
Now

$$\overline{\partial}_w \left( f(w)\!\left( G_r^r \wedge \overline{\partial}_w G^r_\epsilon \right)^{n-1}\right) =
\displaystyle d_w \left( f(w) \,G^r_\epsilon \wedge \left( \overline{\partial}_w G^r_\epsilon \right)^{n-1}\right) $$
\noindent since $f$ is holomorphic and $G^r_\epsilon \wedge \left( \overline{\partial}_w G^r_\epsilon \right)^{n-1}$ is an $(n, n-1)$-form, so Stokes' theorem gives us that 
\begin{equation*}
   \int\limits_{D\setminus \mathcal B^\prime_z}\!\!f(w) F^r(w,z) = \left( \int_I\limits + \int_{II}\limits \right)
   \!f(w)\,\overline\partial_z\! \left( G^r_\epsilon \wedge  \overline{\partial}_w G^r_\epsilon \right)^{n-1} .
\end{equation*}
But the integral over $II$ vanishes since $G^r_\epsilon$ is holomorphic in $z$, for $z \in \mathcal{B}_w$, that is $w \in \mathcal B^\prime_z$. To treat the integral over $I$ we involve the modification $\widehat{G}^r_\epsilon$
that was defined in \eqref{E:def-G-hat-eps-r}.
Then since $I \subset bD$, the same argument as before shows
\begin{equation*}
   \int_I\limits = \int\limits_If(w)\,\overline\partial_z\!\left( \widehat{G}^r_\epsilon \wedge \left( \overline\partial_w\widehat{G}_\epsilon^r \right)^{n-1} \right).
\end{equation*}
However $\widehat{G}^r_\epsilon$ is a generating form over $D$ so the following identity holds for it.
\begin{lemma}\cite[Lemma IV.3.5]{R}
   \begin{equation}\label{E:miracle}
  \overline\partial_z\!\left( \widehat{G}^r_\epsilon \wedge \left( \overline{\partial}_w\widehat G^r_\epsilon \right)^{n-1} \right) =\ d_w (H (w,z))
\end{equation}
where $\displaystyle H$
 is the double form of type
$(0,1)$ in $z$, and of type $(n, n - 2 )$ 

\noindent in  $w$
\begin{equation*}
   H(w, z) = - (n - 1) \,\widehat{G}^r_\epsilon \wedge \left( \overline{\partial}_w \widehat{G}^r_\epsilon \right)^{n-2} \wedge \overline{\partial}_z \widehat{G}^r_\epsilon .
\end{equation*}
\end{lemma}
Using \eqref{E:miracle}, we find
$$
f(w)\,\overline\partial_z\!\left( \widehat{G}^r_\epsilon \wedge \left( \overline\partial_w\widehat{G}_\epsilon^r \right)^{n-1} \right) =
\ \frac{1}{1-n}\,d_w \big(f(w)H (w,z)\big),
$$
since $f$ is holomorphic.
 So by Stokes' theorem
\begin{equation*}
   \int\limits_I f(w)\,\overline\partial_z\!\left( \widehat{G}^r_\epsilon \wedge \left( \overline{\partial}_w \widehat{G}^r_\epsilon \right)^{n-1} \right) = 
   \frac{1}{1-n}\int\limits_{bI}f(w) H(w,z) =\ 0 ,
\end{equation*}
since 
$$
\overline\partial_z\widehat G^r_\epsilon (w, z) = 0\quad \text{for\ any}\quad 
w\in bI\subset b\mathcal B^\prime_z.
$$
As a result,
 $$
   \int\limits_D\! f(w)\,F^r (w, z) = 0 \quad \text{ for } \ z \in \overline{\Omega} \, ,
   $$
   and a limit as $r \to \infty$ gives \eqref{E:C2} in the case when $f\in\vartheta(D)\cap C^1(\overline D)$.
   
To drop the assumption that $f$ is in $C^1(\overline{D})$ we argue as in the proof of Proposition \ref{P:reproducing-B-1-eps}, by replacing the domain $D$ with defining function $\rho$, by the proper sub-domains $D_\lambda$, with defining function $\rho + \lambda$ for $\lambda$ positive.
With $\widetilde{G}^\lambda_\epsilon$ defined as in 
\eqref{E:def-G-tilde-eps-lambda}, we now set
\begin{equation*}
\tilde F^\lambda(w, z)\ =\
\Bigg\{\!\!
\begin{array}{ll}
- \overline{\partial}_z\, \overline{\partial}_w \big( \widetilde{G}^\lambda_\epsilon \wedge \big( \overline{\partial}_w \widetilde{G}^\lambda_\epsilon \big) ^{n-1}\big),
  &\text{if}\ z \in \mathcal{P}_w\\
 \quad\  0, &\text{if}\  z\in \mathcal B_w.
 \end{array}
\end{equation*}
Then as before,
\begin{equation*}
   \int\limits_{D_\lambda}\!f(w)\,\tilde F^\lambda (w, z) = 0\quad 
   \text{for}\quad \lambda<\lambda (z),
\end{equation*}
if is holomorphic in $D$.
A passage to the limit, $\lambda \to 0$, then gives 
\eqref{E:C2} and the proof of Proposition \ref{P:correction} is complete.
\medskip

\noindent For later applications, it will be useful to note that 
$$B_\epsilon (w, z) = B^1_\epsilon (w, z)  + B^2_\epsilon (w, z) $$
 is holomorphic for $z$ in a neighborhood of $\overline{D}$, when every 
(fixed) $w$ is inside $D$. To make this precise, recall the family of domains $D_\lambda =\{\rho<-\lambda\}$ used above, and the positive number $\lambda_0$ appearing in \eqref{E:union}.
\begin{corollary}\label{C:hol-ext-to-nbhd}
   Suppose $0<\lambda < \lambda_0 / 2$, and $w \in D_\lambda$. Then $B_\epsilon (w, z)$ extends to a holomorphic function for $z \in D_{- \lambda}$.
\end{corollary}
\begin{proof}
Let $F(w, z)$ be as in \eqref{E:def-d-bar-data} and $\Omega$ be as in
\eqref{E:Omega-D-lambda}. Note that by
\eqref{E:def-correction} we have
   $$\overline{\partial}_z B^2_\epsilon (w, z) = F(w, z) \quad \text{for}\quad 
    z \in \Omega,$$ 
  and 
    $$F(w, z) = - \overline{\partial}_z B^1_\epsilon (w, z)\quad \text{for}\quad z \in \mathcal{P}_w\,,
    $$
by the definition of $F(w, z)$. Hence $B_\epsilon (w, z)$ extends to a holomorphic function for $z \in \Omega \cap \mathcal{P}_w$. However when $w \in D_\lambda$, we have that $\mathcal{P}_w \supset D_{-\lambda}$ (because $w \in D_\lambda$ and $z \in D_{- \lambda}$ imply that $\rho (z) + \rho (w) < 0 \leq c/2 | z - w |^2$). Also $\Omega \supset D_{- \lambda_0 /2}$ by our construction, 
see \eqref{E:Omega-D-lambda}.  
Thus

\noindent $\Omega \cap \mathcal{P}_w \supset 
(D_{- \lambda}\cap D_{-\lambda_0/2}) = D_{-\lambda}$ for $\lambda < \lambda_0 /2$, and $B_\epsilon (w, z)$ is holomorphic in $D_{- \lambda}$.
\end{proof}

\section{An $L^p$ estimate}\label{S:L-p-estimate}
We prove a basic $L^p$ inequality needed in what follows. We deal with a comparison operator $\varGamma$ defined by
\begin{equation*}
   \varGamma (f)(z) = \int_D\limits  | g (w, z) |^{-n - 1} f(w)\, dV(w) ,\quad z\in D
\end{equation*} 
where $dV$ is the Euclidean volume element in $\mathbb{C}^n$. We also consider the operator $\varGamma_\epsilon$ defined similarly, with $g$ replaced by $g_\epsilon$. 
\begin{proposition}\label{P:comparison-op}
   For $1 < p  < \infty$, \ we have \\
\begin{equation*}
   \| \varGamma (f) \|_{L^p (D)} \leq c_p \| f \|_{L^p (D)}
\end{equation*}
and also, 
\begin{equation*}
   \| \varGamma_\epsilon (f) \|_{L^p (D)} \leq c_p \| f \|_{L^p (D)}
\end{equation*}
with a bound $c_p$ independent of $\epsilon$.
\end{proposition}
Here $\displaystyle{\|f\|_{L^p(D)}}$ denotes the norm 
of $f$ in the Lebesgue space $L^p(D)$ with respect to
 the Euclidean measure in $\mathbb R^{2n}$.
It is clear that by Proposition \ref{P:g-g-eps}, the result for $\varGamma_\epsilon$ is immediately reducible to that for $\varGamma$, which in turn is a consequence of the following.
\begin{lemma}\label{L:comparison-op}
For each $\alpha , \ 0 < \alpha < 1$
\begin{equation}\label{E:int-g-est-z}
   \int\limits_D | g(w, z) |^{-n-1} | \rho (w) |^{-\alpha} dV(w) \leq c_\alpha | \rho (z) |^{- \alpha} .
\end{equation}    
\end{lemma}
Once \eqref{E:int-g-est-z} is proved, the symmetry 
$| g(w, z) | \approx | g (z, w) |$, 
see \eqref{E:g-g-eps-symm}, 
shows that the analogous inequality
\begin{equation}\label{E:int-g-est-w}
   \int\limits_D | g(w, z) |^{-n-1} | \rho (z) |^{- \alpha} dV(z) \leq c^\prime_\alpha | \rho (w) |^{- \alpha}
\end{equation} 
also holds.
The fact that \eqref{E:int-g-est-z} and \eqref{E:int-g-est-w} imply Proposition \ref{P:comparison-op} is a consequence of standard arguments, 
see e.g.  \cite{MS}. Indeed, for fixed $z\in D$ we write 
$$| g (w, z) |^{-n -1} | f(w) | \ =\ F_1(w) \cdot F_2(w)$$ with 
$$F_1 = | g (w, z) |^{-(n+1)/p} | f(w) | | \rho (w) |^{\beta /q}\,;\ \ \  
F_2 = | g (w, z) |^{-(n+1)/q} | \rho (w) |^{-\beta/q},
$$
 where $q$ is the dual exponent to $p, \ 1/p + 1/q = 1$ ($\beta$ will be chosen later). 
 Then by H\"{o}lder's inequality 
\begin{equation*}
   | \varGamma (f)(z) |^p \leq \left( \int | F_1 | \cdot | F_2 | \right)^p \leq \ \int | F_1 |^p \cdot \left( \int |F_2 |^q \right) ^{p/q}
\end{equation*}
Now
\begin{equation*}
  \int | F_2 (w) |^q\,dV(w) = \int | g(w, z) |^{-n+1} | \rho (w) |^{-\beta} dV (w)\ \leq \ c_\beta\, | \rho (z) |^{- \beta}
\end{equation*}
by \eqref{E:int-g-est-z}, if we choose $\beta\in (0, 1)$. Thus
\begin{equation*}
   | \varGamma (f) (z) |^p \leq c \left( \int | g(w, z) |^{-n - 1} | f(w) |^p | \rho (w) |^{\beta p/q} d V (w) \right) |\rho (z) |^{- \beta p/q}.
\end{equation*}
If we integrate this in $z$, interchange the order of integration and use
 \eqref{E:int-g-est-w} with $\alpha = \beta/q$ , (having chosen $\beta$ 
 sufficiently small so that $0 < \beta p/q < 1$ in addition to $0 < \beta < 1)$, then the result is 
\begin{equation*}
   \int_D\limits | \varGamma (f)(z) |^p dV(z) \leq C_p\int\limits_D | f(w) |^p dV (w) \ , 
\end{equation*}
proving the Proposition.
\medskip

We turn to the proof of Lemma \ref{L:comparison-op}. We consider first the heart of the matter, where $z \in D$ is sufficiently close to the boundary, and the region of integration in \eqref{E:int-g-est-z} is limited to those $w$ sufficiently close to $z$. That is, for a small positive constant $c_1$, we show first that 
\begin{equation}\label{E:local-int-g-est-z}
   \int_{| w - z | < c_1}\limits\!\!\!\!\!\!
    | g(w, z)|^{-n-1} |\rho (w) |^{- \alpha}dV(w) \leq c |\rho (z) |^{- \alpha},\  \
\text{for}\ \
| \rho (z) | < c_1.
\end{equation}
Now if $z$ is sufficiently close to $bD$, then the normal projection $\pi (z)$ of $z$ to the boundary is well defined. Hence by a translation and unitary linear transformation of $\mathbb{C}^n$ we can find a new coordinate system $(z_1 , z_2 , \ldots , z_n )$ centered at $\pi (z)$, that is $\pi (z) = (0, \ldots , 0)$, so that if $z_n = x_n + i y_n$, the points $(z_1, \cdots , z_{n-1}, x_n )$ are in the (real) tangent plane at $\pi(z) = (0, \ldots , 0)$; moreover our initial $z$ is given as $(0, \ldots , 0, iy_n )$ in this coordinate system. It follows that 
$$y_n \approx | \rho (z) |;\quad | z - \pi (z) | \approx |\rho (z) |.$$
 In addition 
 $$\displaystyle \frac{\partial\rho}{\partial z_j}  (\pi (z)) = 0, \ 1 \leq j \leq n-1, 
 \ \ \text{while}\ \ 
 \frac{\partial \rho}{\partial y_n} (\pi (z)) \neq 0,
$$
and
$$
\left( \frac{\partial \rho}{\partial z_n} \right)\!( \pi (z))  =  - \frac{i}{2} \left( \frac{\partial \rho}{\partial y_n} \right)\!(\pi (z)).
$$
Thus if $w = (w_1, \cdots , w_n )$, with $w_n = u_n + iv_n$, then 
$$| \Im \langle \partial \rho (\pi (z)), \ w - z \rangle |\ \approx\ | u_n- x_n |\ =\ |u_n |.
$$

Now by Proposition \ref{P:g-g-eps} and the symmetry \eqref{E:g-g-eps-symm}
we have that 
$$| g(w, z) | \approx  | \rho (w) | + | \rho (z) | + 
|\Im \langle \partial \rho (z), w-z \rangle | + | w-z  |^2,$$
and by the above, 
\begin{equation*}
\Im \langle \partial \rho (z), w-z \rangle = 
\Im \langle \partial \rho (\pi (z)), w-z \rangle + \bigO (| \rho (z)|| w - z |)
\end{equation*}
 Combining this with the above and the fact that $|w-z|<c_1$ with $c_1$ small, we obtain
 \begin{equation*}
    | \rho (z)| + | \Im \langle \partial \rho (z), w-z  \rangle | \ 
    \gtrsim\ 
 | \rho (z) | + |u_n |,
 \end{equation*}
and so
\begin{equation}\label{E:abs-g-lower-bound}
| g(w,z) | \gtrsim | \rho (z) | + | \rho (w) | + | u_n |  + | w^\prime |^2 \ ,
\end{equation} 
  where $z^\prime = (z_1 , \ldots, z_{n-1}) = (0, \ldots , 0), \text{ and } \ w^\prime = (w_1, \ldots , w_{n-1})$. 

We now introduce a further coordinate system about $\pi (z) = (0, \ldots , 0)$ via a $C^2$ change of variables so that these new coordinates $(\zeta_1, \zeta_2, \ldots , \zeta_n)$ are related to the previous coordinates by 
$$\zeta_j = z_j , \ \ 1 \leq j \leq n-1;\quad \zeta_n = x_n - i\rho (z).$$
 Thus because of \eqref{E:abs-g-lower-bound}, if we write $-\rho (z) = t$ and $- \rho (w) = s$, then to prove 
 \eqref{E:local-int-g-est-z}
  it suffices to see that 
\begin{equation*}
   \int\limits_{s \in \mathbb R^+}\int\limits_{u_n \in \mathbb R}\,
   \int\limits_{w^\prime \in\mathbb{C}^{n-1}}
    \limits\!\! \frac{s^{- \alpha}}{\left( t + s + | u_n | + | w^\prime |^2 \right)^{n+1}} \, ds\, du_n\, dV (w^\prime ) \ \leq c\, t^{- \alpha}
\end{equation*}
 for any $t > 0$, and $0 < \alpha < 1$.
To prove this inequality, note that by rescaling by the mapping $$w^\prime \to \delta^2 w^\prime; \quad   u_n \to \delta u_n; \quad s \to \delta s,$$
 we are reduced to checking that 
\begin{equation}\label{E:reduced-suff-local-int-g-est-z}
   \int_{\mathbb{R}^+ \times \mathbb{R} \times \mathbb{C}^{n-1}}
   \limits \frac{s^{- \alpha}}{\left( 1 + s + | u_n | + | w^\prime |^2 \right)^{n+1}} \ ds\, du_n\, dV (w^\prime ) \ < \ \infty .
\end{equation} 
Now recall that for $0<\alpha<1$ we have
$$
\displaystyle \int^\infty_0\limits \frac{s^{- \alpha}}{(s + X)^{n+1}}\,ds = c_\alpha X^{-n-\alpha}, \quad \text{whenever}\quad   X > 0
$$ 
(the restriction on $\alpha$ guarantees the convergence of the integral).
 Applying this fact to $X = 1 + | u_n | + | w^\prime |^2$ we see that the integral 
\eqref{E:reduced-suff-local-int-g-est-z} is further reduced to a multiple of 

$$\displaystyle \int_{\mathbb{R} \times \mathbb{C}^{n-1}}\limits
\!\! \frac{1}{\left( 1 + | u_n | + |w^\prime |^2 \right)^{n+ \alpha}} \ du_n\, dV(w^\prime ),
$$ 
and the  finiteness of the latter is easily seen for $\alpha > 0$. 
This proves \eqref{E:local-int-g-est-z}.\\

What remains is the situation where either $| \rho (z) | \geq c_1$ or $| z - w | \geq c_1$ (or both). In any case we note from 
Proposition \ref{P:g-g-eps}
that $| g (w, z) | \geq c^\prime > 0$ and hence it suffices to see that
\begin{equation}\label{E:far-out-int-est-g}
   \int_D\limits | \rho (w) |^{-\alpha} dV(w) < \infty \ , \ \text{ whenever } \ \ \alpha < 1 \ .
\end{equation} 
\indent To verify \eqref{E:far-out-int-est-g} cover the boundary of $D$ by finitely many small balls, so that in each ball we can introduce a coordinate system as above. Then the finiteness of the part of the integral 
\eqref{E:far-out-int-est-g} taken over each such ball is easily reduced the fact that $$\displaystyle \int^1_0\limits s^{- \alpha} ds < \infty,\quad \text{when}\quad \alpha < 1.
$$

\noindent This concludes the proof of Lemma \ref{L:comparison-op} and hence also of Proposition \ref{P:comparison-op}.\\

We now apply this proposition to obtain a preliminary estimate for the operators $\{B^1_\epsilon\}_{\epsilon>0}$. We rewrite $B^1_\epsilon(w, z)$,
the kernel of $B^1_\epsilon$,
 as
\begin{equation}\label{E:repr-B-1-eps-ker}
  B^1_\epsilon (w, z) = \frac{N_\epsilon (w, z)}{g_\epsilon(w, z)^{n+1}}
\end{equation}
 where $N_\epsilon (w, z)$ can be computed from \eqref{E:def-ker-B-1-eps} and 
\eqref{E:def-G-eps}, which give
\begin{equation}\label{E:def-N-eps}
 N_\epsilon(w, z) 
= - 
\left( 
\left( \overline{\partial}_w \eta_\epsilon \right)^{n-1}\!\wedge
\overline{\partial}_w g_\epsilon\wedge
\eta_\epsilon \ +\ 
g_\epsilon\,(\overline\partial_w\eta_\epsilon)^n
\right).
\end{equation}
Here and in the sequel, we use the notation $\bigO_\epsilon$ to indicate 
a form (or a function) whose coefficients are bounded by $c_\epsilon|w-z|$. 
 Now looking back at the definition of $g_\epsilon$ and $\eta_\epsilon$, 
see \eqref{E:def-g-eps} and \eqref{E:def-eta-eps}, 
 we see that 
 \begin{align*}
   &\overline{\partial}_w g_\epsilon (w, z) = - \overline{\partial} \rho (w) + 
   \bigO_\epsilon ( | w- z  |)\\
&\eta_\epsilon (w, z) = \partial \rho (w) + \bigO_\epsilon ( | w-z  | )\\
&\overline{\partial}_w \eta_\epsilon (w, z) = \overline{\partial} \partial \rho (w) + \bigO_\epsilon ( | w-z  | )\, .
\end{align*}
Inserting these expressions in \eqref{E:def-N-eps} yields
\begin{equation}\label{E:repr-N-eps}
N_\epsilon(w, z) =
\left(\overline{\partial} \rho\wedge \partial \rho \wedge \left( \overline{\partial} \partial \rho \right)^{n-1}\right) (w)\, +\, 
\bigO_\epsilon (|w-z|).
\end{equation}
 Since these forms have degree $2n$ in $w$, and the Euclidean volume element in $\mathbb R^{2n}$ may be expressed as
 \begin{equation*}
 dV(w) = \prod\limits_{j=1}^n\frac{i}{2}\,dw_j\,\wedge\,d\overline{w}_j\, ,
 \end{equation*}
 we have
 \begin{equation}\label{E:aux-0}
 \left(\overline{\partial} \rho\wedge \partial \rho \wedge \left( \overline{\partial} \partial \rho \right)^{n-1}\right) \!(w)=\ \mathcal K_0(w)\,dV(w) ,
\end{equation}
where $\mathcal K_0 (w)$ is a continuous function on $\overline D$. We also have
\begin{equation}\label{E:aux-1}
|O_\epsilon (|w-z|)|\ \leq c_\epsilon |w-z|,
\end{equation}
however,
 the bound $c_\epsilon$ may not remain bounded as $\epsilon \to 0$ because
  it
 depends on the first derivatives of $\tau^\epsilon$.
Nevertheless we do have
\begin{corollary}\label{C:B-eps-bdd}
   For each $\epsilon > 0$ and $ 1 < p < \infty$, the operator $B_\epsilon$ is bounded on $L^p (D)$.
\end{corollary}
By Proposition \ref{P:correction}, $B_\epsilon = B^1_\epsilon + B^2_\epsilon$, and the operator $B^2_\epsilon$ has a kernel which is bounded on $\overline{D} \times \overline{D}$, hence it gives rise to bounded operator on $L^p$, for all $1 \leq p \leq \infty$. For the operator $B^1_\epsilon$ we apply 
\eqref{E:repr-B-1-eps-ker} through \eqref{E:aux-1}
 to see that 
 $$\left| B^1_\epsilon (f) (z) \right| \leq (c_0+c_\epsilon) \varGamma_\epsilon ( | f | )(z).
 $$
 Now Proposition \ref{P:comparison-op} grants the boundedness of 
 $B^1_\epsilon$ on $L^p$ for $1 < p < \infty$, proving the Corollary.

\section{The approximate \textquotedblleft symmetry\textquotedblright \ of $B_\epsilon$}
Our goal here is to understand the degree to which the operator $B_\epsilon$ is symmetric. Indeed, if it had exactly that property, that is, $B^\ast_\epsilon = B_\epsilon$, then in view of properties 
(a) and (b) in Proposition \ref{P:correction}, $B_\epsilon$ would necessarily be identical with the Bergman projection. While this is not the case, the facts concerning the approximate symmetry of $B_\epsilon$ are centered in the following key lemma. If 
$$(f, g) = \int_D\limits f(w) \overline{g}(w)\, dV(w)$$
 is the inner product on $L^2(D)$, and $T$ is a bounded operator on 
 $L^2(D)$, we denote by $T^\ast$ its adjoint with respect to this inner product; it satisfies 
 $$\left( T^\ast f, g \right) = \left( f, Tg \right).$$
\begin{lemma}\label{L:B-B-star-dec}
   For each $\epsilon > 0$, we can decompose $B_\epsilon - B_\epsilon^\ast$ as 
\begin{equation}\label{E:B-B-star-dec}
   B_\epsilon - B^\ast_\epsilon = A_\epsilon + C_\epsilon
\end{equation} 
where
\begin{itemize}
\item[(a)] $\left\| A_\epsilon \right\|_{L^p \to L^p} \leq \epsilon\, c_p$ , \quad for  \ \ $1 < p < \infty$\\

\item[(b)]
Each $C_\epsilon$ has a kernel which is continuous on $\overline{D} \times \overline{D}$, and hence $C_\epsilon$ maps $L^1(D)$
 to $C(\overline D)$.
 \end{itemize}
\end{lemma}
In general, the norm of the operator $C_\epsilon$ may increase as $\epsilon \to 0$.
\begin{proof}
   To construct the operators $A_\epsilon$ and $C_\epsilon$ we need to break up $B_\epsilon$ into a part whose kernel is supported near the \textquotedblleft diagonal\textquotedblright \ of $D \times D$, and a complementary part supported away from this set. For this purpose, fix a continuous function $\varphi$ on $\mathbb R^+$, so that $\varphi (t) = 1$, when $0 \leq t \leq \frac{1}{2}$ and $\varphi (t) = 1$ for $t > 1$, and for each $r > 0$ set 
   $$\varphi_r (w, z) = \varphi \left( \frac{| \rho (z) | + | \rho (w) | + | z - w |}{r} \right) $$
We write
\begin{equation*}
   D^r (f)(z) = \frac{1}{(2 \pi i)^n}\int_D\limits 
    \varphi_r (w, z)f(w)B^1_\epsilon (w, z)
    \end{equation*}
    \begin{equation*}
E^r (f)(z) = \frac{1}{(2 \pi i)^n} \int_D\limits(1 - \varphi_r (w, z)) f(w) 
B^1_\epsilon (w, z)
\end{equation*}
so that $B^1_\epsilon = D^r + E^r$. The parameter $r$ will be chosen in terms of $\epsilon$ momentarily.
We then define $A_\epsilon \text{ and } C_\epsilon$ by
$$A_\epsilon = D^r - (D^r )^\ast$$
 and
$$C_\epsilon = E^r - \left( E^r \right)^\ast + B^2_\epsilon - \left( B^2_\epsilon \right)^\ast$$
With these definitions, the identity \eqref{E:B-B-star-dec} clearly holds, since $B_\epsilon = B^1_\epsilon + B^2_\epsilon$.
To proceed further we recall that if $T$ is a bounded operator on $L^2 (D)$ given with a kernel $K$, 
\begin{equation*}
   T(f)(z) = \int_D\limits K(w, z) f(w) dV(w) \ ,
\end{equation*} 
with say $\displaystyle | K (w, z) | \lesssim  | g(w, z) |^{-n-1}$,  then its adjoint $T^\ast$ is given by the kernel $\overline{K} (z, w)$ in place of $K (w, z)$.
In view of the representation \eqref{E:repr-B-1-eps-ker},  the kernel of $D^r$ is
\begin{equation*}
  \frac{(\mathcal K_0(w) + O_\epsilon (|w-z|))\varphi_r(w, z)}{
  (2\pi i)^n\left( g_\epsilon (w, z) \right)^{n+1}},
\end{equation*}
and so to study $D^r - (D^r)^\ast$ we must estimate
\begin{equation*}
   \frac{1}{\left( g_\epsilon (w, z) \right)^{n+1}} \ -\  \frac{1}{\left( \overline{g}_\epsilon z, w) \right)^{n+1}}
   \end{equation*}
   and
   \begin{equation}\label{E:numerator}
   \frac{\mathcal K_0(w)}{(2\pi i)^n} -
   \overline{\left(\frac{\mathcal K_0(z)}{(2\pi i)^n}\right)} + 
    \frac{O_\epsilon (|w-z|)}{(2\pi i)^n} -\overline{\left(\frac{O_\epsilon(|w-z|)}
    {(2\pi i)^n}\right)}
   \end{equation}
     on the support of  $\varphi_r$.
(Note that $\varphi_r (w, z)$ is obviously symmetric in $w$ and $z$). Now
 $$\displaystyle \left( g_\epsilon (w, z) \right)^{-n-1} - \left( \overline{g}_\epsilon (z, w) \right)^{-n-1} = 
   \frac{\left( \overline{g}_\epsilon (z,w) \right)^{n+1} - \left( g_\epsilon (w, z) \right)^{n+1}}{\left( g_\epsilon (w, z) \overline{g}_\epsilon (z, w) \right)^{n+1}}
   $$
   Taking into account the fact that $| g_\epsilon (w, z) | \approx | g_\epsilon (z, w) |$, we see that the above is majorized by a multiple of 
\begin{equation*}
   \frac{\left| g_\epsilon (w, z) - \overline{g}_\epsilon (z, w) \right|}{| g_\epsilon (w, z) |^{n+2}} \ . 
\end{equation*}
In turn, by Propositions \ref{P:g-g-eps} and \ref{P:g-eps-g-eps-conj}, this is majorized by a multiple of 
$$ \frac{\epsilon | z - w |^2}{| g_\epsilon (w, z) |^{n+2}} \lesssim \frac{\epsilon}{| g_\epsilon (w, z) |^{n+1}},$$
whenever $| z - w | < \delta_\epsilon$. 
So if we choose $r$ so that $r \leq \delta_\epsilon$, then 
\begin{equation}\label{E:denominator-est}
   \left| \frac{1}{\left( g_\epsilon (w, z) \right)^{n+1}} - \frac{1}{\overline{g}_\epsilon ( z, w)^{n+1}} \right| \,\leq\, \frac{c\,\epsilon}{| g_\epsilon (w, z) |^{n+1}} 
\end{equation} 
for $(w, z)$ in the support of $\varphi_r$, because then $| z - w | < r$.

\noindent Next, we examine the numerator \eqref{E:numerator}. 
To this end, we first note that 
$$
O_\epsilon (|w-z|) \ =\  \bigO_\epsilon (r)
$$
for $(w, z)$ in the support of $\varphi_r$, because then $| z - w | < r$.
To estimate the difference $\mathcal K_0(w)/(2\pi i)^n -
\overline{\mathcal K_0(z)/(2\pi i)^n}$
we invoke the identities \eqref{E:repr-N-eps} through
\eqref{E:repr-N-eps} along with the computation
\cite[exercise VII.E.7.2]{R}, and obtain
\begin{equation*}
 \frac{\mathcal K_0(w)}{(2\pi i)^n}
 = \frac{n !}{\pi^n}
  | \nabla \rho(w) |^2 \det \mathcal{L}_w 
\end{equation*} 
where $\det \mathcal{L}_{w}$ is the determinant of the Levi form at the point $w$ (for the domain 
$\{ \zeta: \rho (\zeta) < \rho (w) \}$). This shows that
$\mathcal K_0(w)/(2\pi i)^n$ is a \textit{real}-valued continuous function on $\overline{D}$. By its uniform continuity, there is a 
$\delta^\prime = \delta^\prime_\epsilon$, so that 
\begin{equation*}
   \left|\frac{\mathcal K_0(w)}{(2\pi i)^n} -\frac{\mathcal K_0(z)}{(2\pi i)^n}\right| < \epsilon \quad \text{ if } \quad | z - w | <  \delta^\prime_\epsilon .
\end{equation*}

\indent Finally suppose that in the above inequalities we write 
$$\bigO_\epsilon (r) \leq r A_\epsilon$$ 
for an appropriate bound $A_\epsilon$ (depending on $\epsilon$). We can then choose $r$ in terms of $\epsilon$ by taking $\displaystyle r = \min \left( \delta_\epsilon , \ \delta_\epsilon^\prime , \  \epsilon / A_\epsilon \right)$. This guarantees that 
\begin{equation}\label{E:numerator-est}
   \left|\ \frac{\mathcal K_0(w) + O_\epsilon (|w-z|)} 
   {(2\pi i)^n}
   \  -\  
  \overline{\left( \frac{{\mathcal K_0}(z) +
   {O_\epsilon}(|w-z|)}
   {(2\pi i)^n}
  \right)}
   \ \right| 
    \leq c \,\epsilon 
   \end{equation}
  on the support of $\varphi_r$.
Combining \eqref{E:numerator-est} with \eqref{E:denominator-est} then shows that
\begin{equation}
   \left|\ \frac{\mathcal K_0(w) + O_\epsilon (|w-z|)} {(2\pi i)^n\left( g_\epsilon (w, z) \right)^{n+1}}\  -\  
  \overline{\left( \frac{{\mathcal K_0}(z) +
   {O_\epsilon}(|w-z|)}
   {(2\pi i)^n\left({g}_\epsilon (z, w) \right)^{n+1}}\right)}\ \right| \leq \frac{c \,\epsilon}{  \left| g_\epsilon (w, z) \right|^{n+1} } .
\end{equation} 
So by Proposition \ref{P:comparison-op} we conclude that 
\begin{equation*}
   \left\| A_\epsilon \right\|_{L^p \to L^p} = \left\| D^r - (D^r)^\ast \right\|_{L^p \to L^p} \leq \epsilon  c_p ,
\end{equation*}
which is conclusion (a) of this  lemma.\\

The second conclusion is immediate. Indeed since  $| \rho (z) | + | \rho (w) | + | z - w | \geq r/2$ on the support of $1 - \varphi_r$, we have by 
Proposition \ref{P:g-g-eps}
 that $\left| g_\epsilon (w, z) \right| \gtrsim r^2$ there, and hence
 the kernel of $B^1_\epsilon$  is a continuous function of $(w, z)$  there, see
 \eqref{E:repr-B-1-eps-ker}. 

Thus the kernels of $E^r$ and $(E^r)^\ast$ are continuous on $\overline{D} \times \overline{D}$, and by Proposition \ref{P:correction}, the same is true for $B^2_\epsilon$ and $\left( B^2_\epsilon \right)^\ast$. Since $C_\epsilon = E^r - (E^r)^\ast + B^2_\epsilon - \left( B^2_\epsilon \right)^*$, the operator $C_\epsilon$ has a continuous kernel on $\overline{D} \times \overline{D}$. From this it is also evident that $C_\epsilon$ maps $L^1(D)$ to $C(\overline D)$, proving the lemma.
\end{proof}

\section{The main theorems}
Let $B$ be the Bergman projection for the domain $D$. The operator $B$ is the orthogonal projection of $L^2(D)$ to the Bergman space 
$\vartheta L^2 (D)$, the subspace of holomorphic functions in $L^2 (D)$. Then as is well-known, 
\begin{equation}\label{E:def-Bergman-pr}
   B(f)(z) = \int_D\limits B(w, z) f(w) dV(w),\quad f\in L^2(D),\quad z\in D
\end{equation} 
where $B(w, z)$ is the Bergman kernel, which satisfies $B(z, w) = \overline{B}(w, z)$.
\begin{theorem}\label{T:Bergman}
   Suppose the domain is of class $C^2$ and strongly pseudoconvex. Then $f \mapsto B(f)$ extends to a bounded mapping of $L^p (D)$ to itself, for each $p$, \ $1 < p  < \infty$.
\end{theorem}
As is well-known, for each $z \in D$, the function $B(w, z)$ belongs to $L^2 (D)$, see \cite{K}. Hence $B(f)(z)$ is well-defined by 
\eqref{E:def-Bergman-pr} when $f \in L^p (D)$ for $p\geq 2$. When $p \leq 2$, the operator 
\eqref{E:def-Bergman-pr} is initially defined only on $L^2 (D)$, a dense subspace of $L^p (D)$. In either case, the thrust of Theorem 
\ref{T:Bergman} is the inequality 
\begin{equation*}
   \left\| B(f) \right\|_{L^p} \leq c_p \| f \|_{L^p}
\end{equation*} 
for $f$ that belongs to $L^2 (D)$.

There is also a stronger version that involves the operator whose kernel is the absolute value of the Bergman kernel. Define $| B |$ to be the operator
\begin{equation*}
   | B | (f)(z) = \int_D\limits | B (z, w) | f(w) dV(w) ,
\end{equation*} 
(defined initially on $L^2 (D)$).
\begin{theorem}\label{T:abs-Bergman}
   Under the assumptions of Theorem \ref{T:Bergman}, we have that $f \mapsto | B | (f)$ is also bounded on $L^p (D)$,  $1 < p < \infty$.
\end{theorem}
\indent To prove Theorem \ref{T:Bergman} we start with the following identities that hold on $L^2 (D)$. First $B B_\epsilon = B_\epsilon$, which follows from (a) of Proposition \ref{P:correction}, (together with Corollary \ref{C:B-eps-bdd}, for $p=2$); also $B_\epsilon B = B$, which is a consequence of (b) of Proposition \ref{P:correction}. Taking adjoints of the second of these identities and using that $B^\ast = B$ immediately yields
\begin{equation*}
   B \left( I - \left( B^\ast_\epsilon - B_\epsilon \right) \right) = B_\epsilon
\end{equation*} 
It then follows from Lemma \ref{L:B-B-star-dec} that
\begin{equation*}
   B \left( 1 + A_\epsilon \right) = B_\epsilon - B C_\epsilon .
\end{equation*} 
\indent Now for fixed $p, \ 1 < p < \infty$, choose $\epsilon = \epsilon (p)$ so small that according to Lemma \ref{L:B-B-star-dec} we have that $\left\| A_\epsilon \right\|_{L^p \to L^p}  < 1$. Then $I + A_\epsilon$ is invertible as an operator on $L^p$, using a Neumann series. Writing the inverse as 
$\left( I + A_\epsilon \right)^{-1}$ gives us
\begin{equation}\label{E:new-KS}
   B = \left( B_\epsilon - BC_\epsilon \right) \left( I + A_\epsilon \right)^{-1} .
\end{equation} 
\indent Since we know that $B_\epsilon$ is bounded on $L^p$ (see Corollary \ref{C:B-eps-bdd}), 
 it suffices to note that $BC_\epsilon$ is bounded on $L^p$. We observe this first for $p \leq 2$ as follows. The fact that the kernel of $C_\epsilon$ is continuous on 
$\overline D \times \overline D$ gives that $C_\epsilon$ maps $L^1$ to $L^\infty$, and hence $L^p$ to $L^2$. Thus $BC_\epsilon$ maps $L^p$ to $L^2$, and since $p \leq 2$, we have $BC_\epsilon$ maps $L^p$ to $L^p$. This proves that the boundedness of $B$ on $L^p$, when $p \leq 2$. \\

To obtain the result for $p\geq 2$ we may use the (self) duality of $B$ to reduce to the case $p \leq 2$. Alternatively we may retrace the steps leading to \eqref{E:new-KS} to obtain 
\begin{equation}\label{E:other-KS}
   \left( I - A_\epsilon \right) B = B^\ast_\epsilon + C^\ast_\epsilon B .
\end{equation} 
\noindent Then, after inverting $ \left( I - A_\epsilon \right)$, it suffices to see that $C_\epsilon^\ast B$ is bounded on $L^p$. For this one notes that $B$ maps $L^p$ to $L^2$, (since now $L^p \subset L^2$), and $C^\ast_\epsilon$ maps $L^2$ to $L^p$ (since it maps $L^1$ to $L^\infty$). The proof of Theorem \ref{T:Bergman} is now complete.\\

To prove Theorem \ref{T:abs-Bergman} we need to manipulate operators with positive kernels. To do so rigorously we formulate the following definition. Suppose $T$ is a bounded linear operator on $L^p$. We say that $T$ has a \textit{positive majorant} $\widehat{T}$, if  $\widehat{T}$ is bounded linear operator on $L^p$ that satisfies
\begin{equation}\label{E:def-pos-maj}
\begin{cases}
   \widehat{T} (f) \geq 0 \ \ \ \text{ if } \ \ f \geq 0, \ \ \text{ and}\\
| T(f)(z) | \leq \widehat{T}( | f | )(z), \ \ \text{ for } \ a.e. \ \ z.
\end{cases}
\end{equation} 
Observe that if $T_1$ and $T_2$ have positive majorants $\widehat{T}_1 \text{ and } \widehat{T}_2$ respectively, then 
\smallskip

   $\small{\bullet} \ \widehat{T}_1 + \widehat{T}_2$ and $\widehat{T}_1 \circ \widehat{T}_2$ are positive majorants for $T_1 + T_2$ and $T_1 \circ T_2$, 
   
   \quad respectively.
   \medskip
   
\noindent Also suppose $T_n$ have positive majorants $\widehat{T}_n$ and $T_n \to T$, while $\widehat{T}_n \to S$, as $n \to \infty$ strongly in $L^p$, then

   $\bullet\ S$ is a positive majorant of $T$.
\medskip

\noindent As a result, if $T$ has a positive majorant $\widehat{T}$ 
and $\displaystyle \| \widehat{T} \|_{L^p \to L^p} < 1$ \ (hence $\| T \|_{L^p \to L^p} < 1$), then 
\smallskip

 $\bullet \ \big( I - \widehat{T}\, \big)^{-1}$ is a positive majorant of 
 $\big(I - T\,\big)^{-1}$.

Indeed, by the above, $ \sum^N_{n=0}\limits \widehat{T}^n$ is a positive majorant of $\sum^N_{n=0}\limits T^n$, and the assertion follows by a passage to the limit as $N \to \infty$.

Let us now consider the case $p\geq 2 $. We will show first that the Bergman projection has a positive majorant, as an operator on $L^p$. The proof of Corollary \ref{C:B-eps-bdd} shows that $B_\epsilon$ and $B^\ast_\epsilon$ both have positive majorants; (these can be taken to be $c^\prime_\epsilon \varGamma_\epsilon$, for suitable $c^\prime_\epsilon$). Moreover this and the proof of (a) of Lemma \ref{L:B-B-star-dec} also shows that the operator $A_\epsilon$ also has a positive majorant of the form $c_\epsilon \varGamma_\epsilon$. If we set $T = A_\epsilon$, this means that there is a positive majorant $\widehat{T}$ for $T$, with $\| \widehat{T} \|_{L^p \to L^p} < 1$, \ if $\epsilon$ is sufficiently small. The same is true for
$T=(I-A_\epsilon)^{-1}$.

Thus by \eqref{E:other-KS}, $B$ will have a positive majorant as a result of the following simple lemma applied to $ T_0 = C^\ast_\epsilon B$, which 
(by Lemma
 \ref{L:B-B-star-dec}) is bounded from $L^p(D)$ to $C( \overline{D} )$ for all  $p\geq 2$.
\begin{lemma}\label{L:positive-maj}
   Suppose $T_0$ is a bounded linear mapping from $L^p (D)$ to $C(\overline{D} )$. Then as a linear mapping from $L^p (D)$ to $L^p (D)$,  $T_0$ has a positive majorant.
\end{lemma}
\begin{proof}
   For each $z \in D$, consider the linear functional $f \mapsto T_0 (f)(z)$. Then by assumption, $\displaystyle \left| T_0 (f)(z) \right| \leq c \| f \|_{L^p}$, and so there is a $\psi = \psi_z \in L^q (D)$ with $\displaystyle 1/p + 1/q = 1$, \ with $\left\| \psi_z   \right\|_{L^q} \leq c$, such that
\begin{equation*}
   T_0 f(z) = \int_D\limits \psi_z (w) f(w) dV(w).
   \end{equation*}
   Now define $\widehat{T}_0$ by
\begin{equation*}
    \widehat{T}_0 f(z) = \int_D\limits | \psi_z (w) | f(w) dV (w) , 
\end{equation*}
so
\begin{equation*}
   \| \widehat{T}_0 f \|_{L^p} \leq c^\prime \| \widehat{T}_0 f  \|_{L^\infty} \leq c^\prime c \| f \|_{L^p}, 
   \end{equation*}
   \smallskip
   
\noindent by H\"older inequality. 
The lemma is therefore proved.
\end{proof}
As a result $B$, the Bergman projection, has a positive majorant $\widehat{B}$. This means that
\begin{equation}\label{E:B-pos-maj}
    \left| \ \int\limits_D\limits B(w, z) f(w)\  dV(w)\  \right| \ \leq \ \widehat{B} ( | f |)(z)\quad \text{a.\, e.}\ \ z\in D.
\end{equation} 
\noindent To pass from \eqref{E:B-pos-maj} to the operator with kernel
 $| B(w, z) |$ requires a further step, because 
\eqref{E:def-pos-maj} is asserted only for a.e. $z \in D$, and the exceptional set may depend on $f$. To get around this difficulty we modify the
 majorant $\widehat{B}$, replacing it by another majorant $B^\#$ for which the analogue \eqref{E:def-pos-maj} holds  for \textit{all} $z \in D$.

In fact, define the {\it averaging operator} $\mathcal{M}$, by 
\begin{equation*}
   \mathcal{M} (F)(z) = \frac{1}{V(\mathcal{B}_z)} \int_{\mathcal{B}_z}\limits F(w) \,dV(w) \quad z\in D,
\end{equation*}
where $\mathcal{B}_z$ is the ball centered at $z$ of radius equal to 1/2 the distance of $z$ from $bD$. Note that $\mathcal{M}(F)(z)$ is defined for
all $z\in D$, and has the following basic properties:
\begin{itemize}
 \item[\emph{i.}\ ] $\mathcal{M}(B f)(z) = B(f)(z)$ for all $z \in D$ (by the mean value property applied to $B(f)$, since this is holomorphic).
 \smallskip
 
 \item[\emph{ii.}\ ] If $|f(w)|\leq|g(w)|$ a.e.\,$w$, then 
 $\mathcal{M} (|f|)(z)\leq \mathcal{M} (|g|)(z)$ for all $z$.
 
 \item[\emph{iii.}\ ] $\left|\mathcal{M} (f)(z)\right|\leq \mathcal{M} (|f|)(z)$.
 \end{itemize}
 \smallskip
 
 Now define $B^{\#}$ by
 \begin{equation*}
 B^{\#} (f)(z) = \mathcal{M} \left(\widehat{B}( f )\right)(z),\quad z\in D.
 \end{equation*}

Then \eqref{E:def-pos-maj} and the basic properties listed above imply that
\begin{equation}\label{E:B-B-always-maj}
   \left|\ \int_D\limits B (w, z) f(w) dV(w)\ \right| \leq B^\# (| f | )(z)
   \quad \text{for\ all}\quad z\in D.
\end{equation} 

\indent Moreover by the maximal theorem $f \to B^\# (f)$ is also bounded on $L^p$, by virtue of the $L^p$ boundedness of $\widehat{B}$. Finally for fixed $z$ apply \eqref{E:B-pos-maj} to $f(w)$ replaced by 
\begin{equation*}
h(w)=
\left\{
\begin{array}{rcl}
f(w)\displaystyle{\frac{|B(w, z)|}{B(w, z)}},
 &\text{if}&B(w, z)\neq 0 \\
& \\
 0, &\text{if}&  B(w, z) =0
\end{array}
\right.
\end{equation*}
Since the Bergman kernel, $B(w, z)$ is anti-holomorphic in $w\in D$,
then for any fixed $z\in D$ it will vanish only on a zero-measure subset
of $D$. This means that
\begin{equation*}
|h(w)| = |f(w)| \quad \text{a.\ e.}\quad w\in D,
\end{equation*}
and so
\begin{equation*}
B^\#(|h|)(z) \ =\  B^\#(|f|)(z)\quad \text{for\ all}\quad z\in D.
\end{equation*}
Also note that
\begin{equation*}
B(h)(z)\ =\ |B|(f)(z)\quad \text{for\ all}\quad z\in D,
\end{equation*}
and it follows from \eqref {E:B-B-always-maj}
 that
\begin{equation*}
|B(f) (z)|\ \leq\ B^{\#}(|f|)(z)
\quad \text{for\ all}\quad z\in D.
\end{equation*}
From these (and the fact that $B^{\#}$ is a positive majorant for $B$ everywhere in $D$) we obtain
\begin{equation*}
   \left| \  | B | (f)(z) \right| \leq B^{\#} ( | f | )(z), \ \ \text{ for\ all }z \in D .
\end{equation*}
This shows that   $f \to | B | (f)$ is bounded on $L^p(D)$ when $p \geq 2$. Since the operator $| B |$ with kernel $| B(w, z) |$ is obviously self-dual, the usual duality shows that it extends to a bounded operator on $L^p$, for $1 < p \leq 2$. Theorem \ref{T:abs-Bergman} is therefore proved.

\section{Concluding remarks}
\subsection{Density in $\vartheta L^p (D)$}
\begin{proposition}\label{P:density}
   The collection of functions that are each holomorphic in some neighborhood of $\overline{D}$ is dense in $\vartheta L^p (D)$, $1<p<\infty$.
\end{proposition}
\indent In fact suppose that $f \in \vartheta L^p (D)$. Then we know that $f = B_\epsilon (f)$ (see Proposition \ref{P:correction}). Now let 
\begin{equation*}
f_n=
\left\{
\begin{array}{rcl}
f(w),& \text{for}& w \in D_{1/n} =\{\rho<-1/n\},\\
0,&\text{for}&w \in D \setminus D_{1/n}
\end{array}
\right.
\end{equation*}
and set
$F_n = B_\epsilon (f_n )$. Then clearly $\| f_n - f \|_{L^p (D)} \to 0$ as $n \to \infty$, which implies that 
\begin{equation*}
   \left\| F_n - f \right\|_{L^p (D)} = \left\| B_\epsilon (f_n - f) \right\|_{L^p (D)} \leq c \| f_n - f \|_{L^p (D)} \to 0
\end{equation*}
by Corollary \ref{C:B-eps-bdd}.
 However if $1/n < \lambda_{0}/2$, where $\lambda_0$ is as in
\eqref{E:union},
 then
 Corollary \ref{C:hol-ext-to-nbhd} shows that $F_n$
  is holomorphic in $D_{-1/n} \supset \overline{D}$, which proves 
 Proposition \ref{P:density}.
\subsection{The Bergman projection $B^\sigma$}

Suppose $\sigma$ is a positive function on $D$. Define the inner product $( \cdotp , \cdotp )_\sigma$ by 
\begin{equation*}
   (F, G)_\sigma = \int_D\limits F(w) \overline{G} (w) \,\sigma (w) dV(w)
   \end{equation*}
   and write
    $L^2_\sigma(D)$
      for the corresponding $L^2$ space with norm
$$\| F \|_{L^2_\sigma} = ( F, F)^{1/2}_\sigma$$ 
Let $B^\sigma$ be the Bergman projection corresponding to $L^2_\sigma$, that is, the orthogonal projection via $( \cdotp
 , \cdotp )_\sigma$ of $L^2_\sigma$ to the weighted Bergman space $\vartheta L^2_\sigma$. Note that $B = B^\sigma$ if $\sigma \equiv 1$.

We may ask whether $B^\sigma$ (like $B$) is bounded on $L^p$ (or $L^p_\sigma$), for a given $\sigma$. (This question was raised by Polam Yung.) We consider here only the first case of interest, when $\sigma$ is continuous in $\overline{D}$ and nowhere zero. Then we can assert
\begin{itemize}
   \item $B^\sigma$ is bounded on $L^p_\sigma$, \ \ for \ $1 < p < \infty$.
\end{itemize}

Under our assumptions on $\sigma$ the norm on $L^p_\sigma$ is clearly equivalent with the norm on $L^p$. Thus we can also assert 
\begin{itemize}
   \item $B^\sigma$ is bounded on $L^p$, \ \ for \   $1 < p < \infty$.
   \item[]
   \item Similar results hold for the operator $|B^\sigma|$.
\end{itemize}

\indent Note however that these results are not direct consequences of Theorem \ref{T:Bergman}
 since there is no simple relation expressing $B^\sigma$ in terms of $B$.
These assertions can, however, be established by reprising the proofs of Theorems \ref{T:Bergman} and \ref{T:abs-Bergman}. One starts by proving an analogue of 
Lemma \ref{L:B-B-star-dec}, which is based on the following observation. Suppose that $\widetilde{B}^\sigma_\epsilon$ is defined to be the adjoint of $B_\epsilon$ with respect to $( \cdotp , \cdotp )_\sigma$. Then  $\widetilde{B}^\sigma_\epsilon = \sigma^{-1} \cdot B^\ast_\epsilon \cdot \sigma$. Further details are left to the interested reader.

\end{document}